\documentclass[a4paper,12pt,reqno,oneside]{amsart}

\usepackage{setspace} 
\usepackage{typearea} 

\usepackage{amscd,amsmath,amssymb,verbatim}
\usepackage{xr-hyper}
\usepackage{graphicx}
\usepackage{ifthen}
\usepackage{cmll}
\usepackage{paralist,cite}
\usepackage[colorlinks,backref,final,bookmarksnumbered,bookmarks]{hyperref}
\usepackage[backrefs]{amsrefs}
\usepackage{tikz-cd}
\usepackage{enumitem}

\usepackage[utf8]{inputenc}
\usepackage[T1,T2A]{fontenc}

\newcommand{\arxiv}[2][]{\ifthenelse{\equal{#1}{}}
{\href{http://arxiv.org/abs/#2}{\tt arXiv:#2}}
{\href{http://arxiv.org/abs/math/#2}{\tt arXiv:math.#1/#2}}}

\theoremstyle{plain}
\newtheorem{theorem}{Theorem}[section]
\newtheorem{lemma}[theorem]{Lemma}
\newtheorem{corollary}[theorem]{Corollary}
\newtheorem{proposition}[theorem]{Proposition}

\newtheorem{maintheorem}{Theorem}
 
\newtheorem{mainconjecture}{Conjecture}
 
\newtheorem{mainproblem}{Problem}

\theoremstyle{definition}
\newtheorem{example}[theorem]{Example}

\newtheoremstyle{remark}
{}{}{}{}{\itshape}{}{ }{\thmname{#1}\thmnumber{ \itshape #2.}}
\theoremstyle{remark}
\newtheorem{remark}[theorem]{Remark}

\newtheoremstyle{concise}
{}{}{}{}{\bfseries}{}{ }{\thmnumber{#2.}\thmnote{ #3.}}
\theoremstyle{concise}

\newenvironment{embedded roster}[1][0]
{

\begin{enumerate}\setcounter{enumii}{#1}}
{\end{enumerate}}
\makeatletter
\renewcommand\p@enumii{\p@enumi}
\renewcommand{\@listii}{\leftmargin=40pt}
\makeatother

\def\N{\mathbb{N}} 
\def\R{\mathbb{R}} 
\def\Z{\mathbb{Z}} 
\def\Q{\mathbb{Q}}

\def\I{\mathcal{I}}

\def\x{\times}
\def\but{\setminus} 
\def\emb{\hookrightarrow} 

\def\eps{\varepsilon} 
\def\phi{\varphi} 
\def\invlim{\lim\nolimits} 
\def\xr#1{\xrightarrow{#1}} 
 \renewcommand{\:}{\colon}
\DeclareMathOperator{\Fg}{fg} 
 
\DeclareMathOperator*{\colim}{colim} 
\DeclareMathOperator*{\derlim}{lim^1} 
\def\fg{{\Fg}}

\DeclareMathOperator{\im}{im} \DeclareMathOperator{\coker}{coker}
\DeclareMathOperator{\Int}{Int} 
 \DeclareMathOperator{\Cl}{Cl} 
 \DeclareMathOperator{\Hom}{Hom} 
 
 \DeclareMathOperator{\Ext}{Ext} 
 \DeclareMathOperator{\cyl}{cyl} 

\def\invlimfg{\lim\nolimits_\fg} 
\def\derlimfg{\lim\nolimits^1_\fg}

\def\tph#1{\raise2.5pt\hbox{\the\textfont1\char"7F}\!\!#1}
\def\tpm#1{\raise0pt\hbox{\the\textfont1\char"7F}\!#1}
\def\tpl#1{\lower1.5pt\hbox{\the\textfont1\char"7F}\!#1}

\def\bydef{\mathrel{\mathop:}=}

\begin{document}

\title{Lim colim versus colim lim. I}
\author{Sergey A. Melikhov}
\address{Steklov Mathematical Institute of Russian Academy of Sciences,
ul.\ Gubkina 8, Moscow, 119991 Russia}
\email{melikhov@mi-ras.ru}

\begin{abstract}
We study a model situation in which direct limit ($\colim$) and inverse limit ($\lim$) do not commute,
and offer some computations of their ``commutator''.

The homology (i.e.\ the Steenrod--Sitnikov homology) of a separable metrizable space $X$ has two well-known 
approximants: $qH_n(X)$ (``\v Cech homology'') and $pH_n(X)$ (``\v Cech homology with compact supports''), 
which are not homology theories but are nevertheless interesting as they are $\lim\,\colim$ and $\colim\,\lim$ 
applied to homology of finite simplicial complexes.
The homomorphism $pH_n(X)\xr{\tau_X} qH_n(X)$, which is a special case of the natural map 
$\colim\,\lim\to\lim\,\colim$,
need not be either injective (P. S. Alexandrov, 1947) or surjective (E. F. Mishchenko, 1953), but its surjectivity
for locally compact $X$ remains an open problem. 
In the case $n=0$ we obtain an affirmative solution of this problem.

For locally compact $X$, the dual map in cohomology $pH^n(X)\to qH^n(X)$ is shown to be surjective and
its kernel is computed, in terms of $\lim^1$ and a new functor $\lim^1_{\text{fg}}$.
The original map $\tau_X$ is surjective and its kernel is computed when $X=S^m\but Y$, where $Y$ is locally compact,
and more generally when $X$ is a ``coronated polyhedron'', i.e.\ contains a compactum whose complement is a polyhedron.

\end{abstract}

\maketitle
\section{Introduction}

The use of $\lim^1$ (and in extreme cases also $\lim^2,\lim^3,\dots$) provides a reasonable description of
limiting behavior in homology and cohomology for metric simplicial complexes (or for CW complexes) and, 
on the other hand, for compact spaces.
In contrast, homology and cohomology (even ordinary) of non-triangulable non-compact spaces have been quite poorly 
understood until recently, due to the lack of any clues on how direct limits interact with inverse limits. 
Some recent progress in algebraic topology of metrizable spaces \cite{M-I}, \cite{M-Ia}, \cite{M-II} is based to
a large extent on certain forms of commutativity between direct and inverse limits.
This commutativity is quite literal in \cite{M-II}, where the axiomatic (co)homology of a Polish space is 
expressed as the (co)homology of a certain cellular (co)chain complex, based on (co)chains ``that are finite in 
one direction but possibly infinite in the other direction''.
Two natural approaches to describing such ``semi-infinite'' (co)chains: as a direct limit of inverse limits and
as an inverse limit of direct limits --- turn out to lead to the same answer.
A non-abelian version of this commutativity phenomenon (with homotopy in place of homology) is the starting point 
in the development of fine shape \cite{M-I}.

The present series ``lim colim versus colim lim'' is devoted to forms of commutativity (and measurement of 
non-commutativity) between direct and inverse limits in a specific model situation, which is of interest 
for algebraic topology of separable metrizable spaces.%
\footnote{Thus we do not intend to cover a general algebraic setting here.
The reader who is interested to see examples of how direct and inverse limits of abelian groups commute 
and don't commute is invited to look at \cite{M00}*{\S\ref{book:prod-lim} and \S\ref{book:colimlim-limcolim}}, 
which are just about that.
Closely related (purely algebraic) issues are also addressed in \cite{M-Ia}.}

Namely, given a metrizable space $X$,
\begin{enumerate}
\item[($\alpha\beta$)] we may first approximate $X$ by compacta $K_\alpha$ and then approximate each $K_\alpha$ by 
finite simpicial complexes $Q_{\alpha\beta}$ (that is, the $K_\alpha$ are all compact subsets of $X$ and
the $Q_{\alpha\beta}$ are the nerves of all open covers of $K_\alpha$); or else
\item[($\beta\alpha$)] we may first approximate $X$ by simplicial complexes $P_\beta$ and then approximate each 
$P_\beta$ by finite simplicial complexes $Q_{\alpha\beta}$ (that is, the $P_\beta$ are the nerves of 
all open covers of $X$ and the $Q_{\alpha\beta}$ are all finite subcomplexes of $P_\beta$).
\end{enumerate}
This leads, in particular, to natural homomorphisms
\[\colim_\alpha\lim_\beta H_n(Q_{\alpha\beta})\xr{\tau_n^X}\lim_\beta\colim_\alpha H_n(Q_{\alpha\beta})\]
\[\colim_\beta\lim_\alpha H^n(Q_{\alpha\beta})\xr{\tau^n_X}\lim_\alpha\colim_\beta H^n(Q_{\alpha\beta})\]
(these will be introduced in a more systematic way in \S\ref{tau} after a preparation in \S\ref{background}) 
and to the following

\begin{mainproblem}\label{prI}
Compute the kernels and cokernels of $\tau_n^X$ and $\tau^n_X$.
\end{mainproblem}

Apart from the straightforward condition that $\tau_n^X$ and $\tau^n_X$ be isomorphisms, there is also a more 
subtle idea of commutativity between direct and inverse limit in our topological setting, which takes into account 
``correction terms'' for the inverse limit: the derived limits $\lim^p$, $p>0$.
It concerns the Steenrod--Sitnikov homology and the \v Cech cohomology of the separable metrizable space $X$,
which are defined as follows:
\[H_n(X)\bydef \colim_\alpha H_n(K_\alpha)\]
\[H^n(X)\bydef \colim_\beta H^n(P_\beta)\]
where $H_n(K_\alpha)$ is Steenrod homology (concerning the latter see e.g.\ \cite{M1} or \cite{M00}).

\begin{mainproblem} \label{prII}
Can one go in the opposite direction? 
That is,

(a) can $H_n(X)$ be reconstructed from $\lim H_n(P_\beta)$ and $\lim^p H_q(P_\beta)$ for $p>0$?

(b) can $H^n(X)$ be reconstructed from $\lim H^n(K_\alpha)$ and $\lim^p H^q(K_\alpha)$ for $p>0$?
\end{mainproblem}

Here ``reconstructed'' is meant in a sense slightly weaker than outright computation, so that, for instance, groups
$G_n$ are understood to be reconstructible from the term $E^2_{pq}$ of a spectral sequence converging to $G_{p+q}$.

To see the connection between Problems \ref{prI} and \ref{prII}, let us note that 
\[H_n(P_\beta)\simeq\colim_\alpha H_n(Q_{\alpha\beta})\]
\[H^n(K_\alpha)\simeq\colim_\beta H^n(Q_{\alpha\beta})\]
and there are short exact sequences (Milnor et al., see Theorem \ref{milnor-ses'})
\[0\to\derlim_\beta H_{n+1}(Q_{\alpha\beta})\to H_n(K_\alpha)\to\lim_\beta H_n(Q_{\alpha\beta})\to 0\]
\[0\to\derlim_\alpha H^{n-1}(Q_{\alpha\beta})\to H^n(P_\beta)\to\lim_\alpha H^n(Q_{\alpha\beta})\to 0.\]

\subsection{On Problem II}
For {\it local compacta}, i.e.\ locally compact separable metrizable spaces, Problem \ref{prII}(b) is known to have 
affirmative solution since the 1970s.
Indeed, if $X$ is a local compactum, or equivalently its compact subsets admit a cofinal subsequence,%
\footnote{Every local compactum $X$ is the union of its compact subsets 
$K_1\subset K_2\subset\dots$, where each $K_i\subset\Int K_{i+1}$.
These constitute a cofinal sequence in the poset of all compact subsets of $X$.
Conversely, if $X$ is a separable metrizable space whose poset of compact subsets contains a cofinal sequence,
then $X$ is easily seen to be a local compactum. 
See \cite{M00}*{Proposition \ref{book:local compactum}} for the details.}
then there is Petkova's short exact sequence (see Theorem \ref{10.10'})
\[0\to\derlim_\alpha H^{n-1}(K_\alpha)\to H^n(X)\to\lim_\alpha H^n(K_\alpha)\to 0.\]

No result of this kind in homology was known until recently.
In fact, the nerves of a metrizable space $X$ admit a cofinal subsequence if and only if $X$ is the union
of a compactum and a set of isolated points of $X$ (see \cite{M-V}*{Theorem \ref{cor:uc}}).
But isolated points make no contribution to homology of positive dimension, nor even to the $\derlim$
part of $0$-homology; so we do not get anything interesting in this way.

However, it turns out that there exists a much wider class of spaces that are still ``dual'' in a way to local compacta.
A {\it coronated polyhedron} is a metrizable space that contains a compactum whose complement is homeomorphic to 
a polyhedron (i.e.\ a simplicial complex with the metric topology).
Coronated polyhedra include, apart from compacta and polyhedra, many familiar spaces, such as the topologist's 
sine curve (and the Warsaw circle) and the comb space (also known as the comb-and-flea space).
The complement of a locally compact subset of a sphere is always a coronated polyhedron \cite{M-V}.

Although the directed set of open covers of a coronated polyhedron $X$ need not admit a cofinal subsequence, 
this turns out to be ``true up to homotopy'', as $X$ admits a sequential polyhedral resolution \cite{M-V}.
Due to this (but not only this), there is a short exact sequence in homology 

\begin{theorem}\label{co-Petkova} \cite{M-V} 
For any coronated polyhedron $X$ and any its sequential ANR resolution $\dots\to P_2\to P_1$ 
there is a short exact sequence
\[0\to\derlim_\beta H_{n+1}(P_\beta)\to H_n(X)\to\lim_\beta H_n(P_\beta)\to 0.\]
\end{theorem}

This yields an affirmative solution to Problem \ref{prII}(a) for coronated polyhedra.

It is likely, however, that Problem \ref{prII}(a) for local compacta and Problem \ref{prII}(b) for
coronated polyhedra admit no satisfactory solution in ZFC.

\begin{example} \label{MP}
Let $X=(\R^n\x\N)^+\x\N$, where $\N$ denotes the countable discrete space and $+$ stands for 
the one-point compactification.
Let us note that $X$ is a local compactum.
It is not hard to compute that $H_n(X)\simeq\bigoplus_{i\in\N}\prod_{i\in\N}\Z$ and
the reduced homology $\tilde H_i(X)=0$ for $i\ne n$ (see \cite{M-II}*{Example \ref{axi:10.3}(a)}).

Let $N_\beta$ run over the nerves of all open covers of $X$, and let $G_\beta=H_n(N_\beta)$.
See \cite{M-II}*{Examples \ref{axi:10.3}, \ref{axi:non-closed}} for a detailed discussion of 
the inverse system $G_\beta$.
It is well-known from a paper by S. Marde\v si\' c and A. Prasolov \cite{MP} that the statement 
\[\lim\nolimits^1 G_\beta=0\] 
cannot be either proved or disproved in ZFC.
More precisely, $\lim^1 G_\beta=0$ follows from the Proper Forcing Axiom \cite{DSV}*{\S3}, as well as from
other set-theoretic assumptions \cite{DSV}*{\S4}, whereas $\lim^1 G_\beta\ne 0$ follows from 
the Continuum Hypothesis \cite{MP} as well as from other set-theoretic assumptions 
\cite{DSV}*{\S2}, \cite{To}.
\end{example}

\begin{example} \label{MP'}
Let $Y$ be the metric wedge $\bigvee_{i\in\N} (\R^n\x\N)^+$, or equivalently the metric quotient 
$X/(\{\infty\}\x\N)$, where $X$ is as in the previous example 
(see \cite{M00}*{\S\ref{book:corrected} and \S\ref{book:metric wedge}} for the definitions 
of a metric wedge and a metric quotient, and \cite{M00}*{Corollary \ref{book:nullseq-sqcup}} 
for the proof of the equivalence for $n=0$; the argument for $n>0$ is similar).
Let us note that $Y$ is a non-locally-compact coronated polyhedron.
It is not hard to compute that $H^n(Y)\simeq\bigoplus_{i\in\N}\prod_{i\in\N}\Z$ and
the reduced cohomology $\tilde H^i(Y)=0$ for $i\ne n$ (see \cite{M-II}*{Example \ref{axi:10.3}(b)}).

Let $K_\alpha$ run over all compact subsets of $Y$, and let $H_\alpha=H^n(K_\alpha)$.
The inverse systems $H_\alpha$ and $G_\beta$ have isomorphic cofinal subsets 
(see \cite{M-II}*{Example \ref{axi:10.3}}).
Then it follows from the previous example that that the statement $\lim^1 H_\alpha=0$
also cannot be either proved or disproved in ZFC.
\end{example}

Part II of the present series \cite{M-IV} establishes an affirmative solution to Problems \ref{prII}(a)
(for finite-dimensional separable metrizable spaces) and \ref{prII}(b) (for arbitrary separable metrizable spaces) 
in ZFC, but in a modified form: the derived limits, including the inverse limit, have to be ``corrected'' 
so as to take into account a natural topology on the indexing sets.
The modified derived limits coincide with the usual ones when the indexing sets are discrete.
Keeping track of the topology on the indexing sets is also implicit in the approach of \cite{M-I}.

\subsection{On Problem I}
Let us now return to the task of understanding the homomorphisms $\tau_n^X$ and $\tau^n_X$.
In \S\ref{mainsection} we introduce a new functor $\derlimfg$ and observe that a natural transformation
between $\derlimfg$ and $\derlim$ is always injective.
This enables some progress on Problem \ref{prI} (see Theorems \ref{10.13} and \ref{10.13'}):

\begin{maintheorem} \label{thA}
(1) If $X$ is a coronated polyhedron, then $\tau_n^X$ is surjective and its kernel is isomorphic to 
$\derlim H_{n+1}(P_\beta)/\derlimfg H_{n+1}(P_\beta)$.

(2) If $X$ is a local compactum, then $\tau^n_X$ is surjective and its kernel is isomorphic to 
$\derlim H^{n-1}(K_\alpha)/\derlimfg H^{n-1}(K_\alpha)$.
\end{maintheorem}

The proofs of (1) and (2) are largely independent of each other and are based on rather different 
ideas and constructions.

The following result is relatively easy (see Theorem \ref{nobeling-los} and Proposition \ref{0-cohomology}).

\begin{maintheorem} \label{thB}
(1) If $X$ is a local compactum, then $\tau_0^X\:pH_0(X)\to qH_0(X)$ is surjective.

(2) $\tau^0_X\:pH^0(X)\to qH^0(X)$ is an isomorphism for every metrizable space $X$.
\end{maintheorem}

The use of the chain complexes for Steenrod--Sitnikov homology introduced in \cite{M-II} enables some 
further progress, or at least so it seems (see Corollary \ref{lim2corollary}):

\begin{maintheorem} \label{thC} 
Let $X$ be a local compactum.
Assume that 

(1) $\lim^1 G_\beta=0$, where $G_\beta$ is the Marde\v si\'c--Prasolov inverse system (see Example \ref{MP}),

(2) $\lim^2 H_{n+1}(N_\alpha)=0$, where $N_\alpha$ runs over the nerves of all open covers of $X$, and

(3) $\lim^3 H_\beta=0$ for every inverse system $H_\beta$ of cofinality at most continuum.

Then $\tau_n^X$ is surjective.
\end{maintheorem}

Here assertion (1) is independent of ZFC (see Example \ref{MP}) and assertion (3) is independent of ZFC 
(see \S\ref{lim2van}).
However, the author does not know if they are consistent with each other.%
\footnote{J. Bergfalk pointed out to the author that they do appear to contradict each other:
``if $\lim^1$ of the Marde\v si\v c--Prasolov system is $0$ then the continuum is of cardinality at least $\aleph_2$. 
And under these circumstances, there will exist inverse systems of the type appearing in (2) 
(and of cofinality $\aleph_2$, in particular) whose $\lim^3$ is nonzero.''
Nevertheless it might still be better to mention Theorem \ref{thC} than to keep it in secret, 
as one may be able to extract a more convincing result from its proof.}

As for assertion (2), the author does not know whether there exists a local compactum $X$ that does not satisfy (2).
It would be quite surprising if someone manages to construct such a space.
By contrast, it seems much more likely that there exists a Polish space $X$ that does not satisfy (2), although 
a rather natural approach to constructing it yields only a somewhat weaker result 
(Example \ref{lim2example}).

\section{Background} \label{background}

\subsection{Expansions}

Let us recall that an {\it expansion} of a metrizable space $X$ is an inverse system 
$(P_\lambda,[p^\lambda_\mu];\Lambda)$ in the homotopy category (thus the bonding maps 
$p^\lambda_\mu\:P_\lambda\to P_\mu$ are required to commute only up to homotopy) along
with maps $p^\infty_\lambda\:X\to P_\lambda$ which commute up to homotopy with the bonding maps
and satisfy the following two conditions for every polyhedron $K$ 
(or equivalently for every ANR $K$):
\begin{enumerate}
\item[(E1)] Given a map $f\:X\to K$, there exists a $\lambda\in\Lambda$ and a map 
$f'_\lambda\:P_\lambda\to K$ such that the composition 
$f'\:X\xr{p^\infty_\lambda}P_\lambda\xr{f'_\lambda}K$ is homotopic to $f$.
\item[(E2)] Given a $\kappa\in\Lambda$ and maps 
$P_\kappa\overset{f_\kappa}{\underset{g_\kappa}{\rightrightarrows}}K$ such that 
the compositions $X\xr{p^\infty_\kappa}P_\kappa
\overset{f_\kappa}{\underset{g_\kappa}{\rightrightarrows}}K$ are homotopic, there exists 
a $\lambda\ge\kappa$ such that the compositions $P_\lambda\xr{p^\lambda_\kappa}P_\kappa
\overset{f_\kappa}{\underset{g_\kappa}{\rightrightarrows}}K$ are homotopic.
\end{enumerate}
If $Y$ is a closed subset of $X$, an {\it expansion} of $(X,Y)$ is an inverse system 
$\big((P_\lambda,Q_\lambda),[p^\lambda_\mu];\Lambda)$ in the homotopy category of pairs along
with maps $p^\infty_\lambda\:(X,Y)\to (P_\lambda,Q_\lambda)$ which commute up to homotopy with 
the bonding maps and satisfy the similar conditions for every polyhedral pair $(K,L)$,
or equivalently for every closed pair of ANRs $(K,L)$.

If additionally each $P_\lambda$ is a metric polyhedron (resp.\ ANR), and each $Q_\lambda$ 
is its closed subpolyhedron (resp.\ a closed subset which is an ANR), then the expansion 
is referred to as a {\it polyhedral expansion} (resp.\ {\it ANR expansion}).

\begin{example} The inverse system given by the nerves $N_C$ of all open covers $C$ of $X$
is a polyhedral expansion of $X$.
The inverse system of pairs $(N_C,N_{C|_Y})$, where $C|_Y$ is the cover of $Y$ by 
those intersections of the elements of $C$ with $Y$ that are nonempty, is a polyhedral expansion of $(X,Y)$.
\end{example}

\begin{example} Suppose that $X$ is embedded in an absolute retract $M$.
For instance, $M$ can be taken to be the vector space of all bounded functions $X\to\R$ 
with the sup norm; or if $X$ is separable, $X$ can be taken to be the Hilbert cube;
or if $X$ is separable and finite-dimensional, $X$ can be taken to a finite-dimensional cube.

Then the inverse system given by all open neighborhoods of $X$ in $M$ and their inclusions
in each other is an ANR expansion of $X$.
The inverse system of all pairs $(U,V)$, where $U$ is a neighborhood of $X$ and $V\subset U$
is a neighborhood of $Y$, is an ANR expansion of $(X,Y)$.
\end{example}

The details of these examples and all necessary background information can be found in \cite{M00} 
or \cite{MS}.

\subsection{Quasi-(co)homology}
We define the {\it quasi-homology} $qH_n(X)$ of a metrizable space $X$ to be the inverse limit 
$\invlim_\alpha H_n(P_\alpha)$, where the $P_\alpha$ form an ANR expansion of $X$. 
This does not depend on the choice of the ANR expansion since any two ANR expansions of $X$
are isomorphic in pro-homotopy (see \cite{M00}*{Corollary \ref{book:expansion3}}).
More generally, for a closed $Y\subset X$ we set \[qH_n(X,Y)\bydef \invlim_\alpha H_n(P_\alpha,Q_\alpha),\] 
where $(P_\alpha,Q_\alpha)$ form an ANR expansion of $(X,Y)$.

Dually, the {\it quasi-cohomology} $qH^n(X)$ of $X$ is defined as the inverse limit of \v Cech cohomology
$\invlim_K H^n(K)$ over all compacta $K\subset X$.
More generally, for a closed $Y\subset X$ we set \[qH^n(X,Y)\bydef \invlim_K H^n(K,\,K\cap Y).\]

Quasi-homology has traditionally been called ``\v Cech homology'' (or ``Alexandroff--\v Cech homology''), 
which is a traditionally confusing terminology since it is not a homology theory (it does not satisfy 
the exact sequence of a pair, already on pairs of compacta, see e.g.\ \cite{ES}*{\S X.4 ``Continuity versus 
exactness''}).
Also quasi-cohomology, which is similarly not a cohomology theory, is not to be confused with 
\v Cech cohomology (which is a cohomology theory).

Apart from the exactness axiom, quasi-(co)homology does satisfy all other axioms of an ordinary (co)homology theory, 
and even more: 
\begin{itemize}
\item the map excision axiom and in particular the usual excision; 
\item (anti)shape invariance and in particular homotopy invariance; 
\item additivity with respect to disjoint unions and with respect to metric wedges
\end{itemize}
(see \cite{M00}*{proof of Theorems \ref{book:homology-ext} and \ref{book:cohomology-ext}, but using 
Theorem \ref{book:coprod-lim} instead of Theorem \ref{book:prod-colim}}).
Let us note that $\bigsqcup$-additivity of quasi-homology was proved originally in \cite{MP}*{Theorem 9}, but 
the proof mentioned above is easier.
Let us also note that quasi-homology and quasi-cohomology are fine shape invariant by
\cite{M-I}*{either Theorem \ref{fish:main5} or Theorem \ref{fish:main1a}}.

From the Alexander duality for compacta (see \cite{M00}*{Theorem \ref{book:alex duality}(a)}) we get,
by applying inverse limit to both sides:

\begin{theorem} \label{alex duality'}
Let $X$ be a nonempty proper subset of $S^m$. 
Then $q\tilde H_n(S^m\but X)\simeq q\tilde H^{m-n-1}(X)$.
\end{theorem}

Here $q\tilde H_n(Y)=\ker\big(qH_n(Y)\to qH_n(pt)\big)$ and $q\tilde H^k(Y)=\coker\big(qH^k(pt)\to qH^k(Y)\big)$.

\subsection{Milnor-type short exact sequences}

For a metrizable $X$ and its closed subset $Y$ let 
\[q^kH_n(X,Y)\bydef \invlim^k H_n(P_\alpha,Q_\alpha),\]
where the pairs $(P_\alpha,Q_\alpha)$ form an ANR expansion of $X$, and let
\[q^kH^n(X,Y)\bydef \invlim^k H^n(K,\,K\cap Y)\]
over all compact $K\subset X$.
When $k=0$, this repeats the definition of quasi-(co)homology.
For $k>0$, it can be called {\it phantom quasi-(co)homology.}

Milnor's original short exact sequences (see \cite{M00}*{Corollary \ref{book:milnor-ses}(a,b)})
yield the following:

\begin{theorem} \label{milnor-ses'} (a) For every compactum $K$ there is a natural short exact sequence 
\[0\to q^1H_{n+1}(K)\to H_n(K)\to qH_n(K)\to 0.\]

(b) For every countable simplicial complex $L$ there is a natural short exact sequence
\[0\to q^1H^{n-1}(|L|)\to H^n(|L|)\to qH^n(|L|)\to 0.\]
\end{theorem}

The naturality is with respect to arbitrary maps (using PL approximation for (b)).

\begin{proof}[Proof. (a)]
Let us represent $K$ as the limit of an inverse sequence of compact polyhedra $P_i$.
Then $P_i$ form a polyhedral expansion of $K$ (see \cite{M00}*{Lemma \ref{book:compact-resolution} and 
Theorem \ref{book:exp-res}}).
Hence $qH_n(K)\simeq\lim H_n(P_i)$ and $q^1H_{n+1}(K)\simeq\lim^1 H_{n+1}(P_i)$, and the
assertion becomes a reformulation of \cite{M00}*{Corollary \ref{book:milnor-ses}(a)}.
\end{proof}

\begin{proof}[(b)] Let us fix an order on the simplexes of $L$ (by natural numbers) and let $K_i$ be 
the union of the first $i$ simplexes.
Then $qH^n(|L|)\simeq\lim H^n(K_i)$ and $q^1H^{n-1}(|L|)\simeq\lim^1 H^{n-1}(K_i)$
(see \cite{M00}*{proof of Corollary \ref{book:poly-cofinal2}}).
Now the assertion becomes a reformulation of \cite{M00}*{Corollary \ref{book:milnor-ses}(b)}.
\end{proof}

\begin{remark} \label{lim2fg}
In fact, (b) is also valid for uncountable simplicial complexes 
\cite{AY}*{Corollary 12} and (a) for compact Hausdorff spaces \cite{KS}*{Theorems 3(1) and 5(1)} 
(see also \cite{Sk5}*{end of \S8.1.2}), \cite{Sk84}*{Theorem 3.2} --- but for a rather different reason: 
$\invlim^2 G_\alpha=0$ if each $G_\alpha$ is a finitely generated abelian group 
\cite{Roo2}*{Theorem 2} (see also \cite{Ku}*{Theorem 5(c)}, \cite{Je}*{Remark on p.\ 65}, \cite{Yo}, 
\cite{HM}, \cite{Mard}*{Theorem 20.17}).

Theorem \ref{milnor-ses'}(b) is also valid for locally compact ANRs, since their compact subsets have 
finitely generated cohomology groups upon restriction to any compact subset of the interior 
(see \cite{M1}*{Theorem 6.11(b)}).
\end{remark}

Petkova's short exact sequence (see \cite{M00}*{Corollary \ref{book:milnor-ses}(c)}) yields the following:

\begin{theorem} \label{10.10'}
For every local compactum $X$ there is a natural short exact sequence
\[0\to q^1H^{n-1}(X)\to H^n(X)\to qH^n(X)\to 0.\]
\end{theorem}

The hypothesis of local compactness cannot be dropped by Example \ref{10.5'''} 
and Corollary \ref{epi-epi}(a) below.

From \cite{M-V}*{Theorem \ref{cor:main-ses}} we obtain:

\begin{theorem} \label{co10.10'}
For every coronated polyhedron $X$ there is a natural short exact sequence
\[0\to q^1H_{n+1}(X)\to H_n(X)\to qH_n(X)\to 0.\]
\end{theorem}

The hypothesis that $X$ is a coronated polyhedron cannot be dropped, by Example \ref{10.5''} and 
Corollary \ref{epi-epi}(b) below.

\subsection{Pseudo-(co)homology}
We define the {\it pseudo-homology} $pH_n(X)$ of the metrizable space $X$ to be the direct limit 
of the quasi-homology groups, $\colim_K qH_n(K)$, over all compacta $K\subset X$.
More generally, for a closed $Y\subset X$ we set \[pH_n(X,Y)\bydef \colim\nolimits_K qH_n(K,\,K\cap Y).\]
Traditionally, pseudo-homology has been called ``\v Cech homology with compact supports''.

Dually, the {\it pseudo-cohomology} $pH^n(X)$ is defined as the direct limit of the quasi-cohomology groups,
$\colim_\alpha qH^n(P_\alpha)$, where the $P_\alpha$ form an ANR expansion of $X$. 
This does not depend on the choice of the ANR expansion since any two ANR expansions of $X$
are isomorphic in pro-homotopy (see \cite{M00}*{Corollary \ref{book:expansion3}}).
More generally, for a closed $Y\subset X$ we set \[pH^n(X,Y)\bydef \colim\nolimits_\alpha qH^n(P_\alpha,Q_\alpha),\]
where $(P_\alpha,Q_\alpha)$ form an ANR expansion of $(X,Y)$.

Pseudo-(co)homology again satisfies all axioms of an ordinary (co)homology theory except for exactness, 
and more: 
\begin{itemize}
\item the map excision axiom and in particular the usual excision; 
\item (anti)shape invariance and in particular homotopy invariance; 
\item additivity with respect to disjoint unions and with respect to metric wedges
\end{itemize}
(see \cite{M00}*{proof of Theorems \ref{book:homology-ext} and \ref{book:cohomology-ext}}).
Let us note that pseudo-homology and pseudo-cohomology are also fine shape invariant by
\cite{M-I}*{either Theorem \ref{fish:main5} or Theorem \ref{fish:main1a}}.

Theorem \ref{alex duality'} implies, by applying direct limit to both sides, that for any $X\subset S^m$
\[p\tilde H_n(S^m\but X)\simeq p\tilde H^{m-n-1}(X),\]
where $p\tilde H_n(Y)=\ker\big(pH_n(Y)\to pH_n(pt)\big)$ and $p\tilde H^k(Y)=\coker\big(pH^k(pt)\to pH^k(Y)\big)$.

For a metrizable $X$ and a closed $Y\subset X$ let
\[p^kH_n(X,Y)\bydef \colim q^kH_n(K,\,K\cap Y)\] over all compact $K\subset X$,
and \[p^kH^n(X,Y)\bydef \colim q^kH^n(P_\alpha,Q_\alpha)\]
where the pairs $(P_\alpha,Q_\alpha)$ form an ANR expansion of $X$.
When $k=0$, this repeats the definition of pseudo-(co)homology.
For $k>0$, it can be called {\it phantom pseudo-(co)homology.}
 
Since direct limit is an exact functor, Theorem \ref{milnor-ses'} yields

\begin{corollary} \label{milnor-ses''}
For every metrizable space $X$ there are natural short exact sequences
\begin{gather*}
0\to p^1H_{n+1}(X)\to H_n(X)\to pH_n(X)\to 0,\\
0\to p^1H^{n-1}(X)\to H^n(X)\to pH^n(X)\to 0.
\end{gather*}
\end{corollary}

\section{Alexandroff's homomorphism} \label{mainsection}

\subsection{Alexandroff's homomorphism} \label{tau}

By the universal property of direct limits, the inclusion induced homomorphisms $qH_n(K)\to qH_n(X)$, where
$K$ runs over compact subsets of $X$, factor through a unique homomorphism \[\tau_n^X\:pH_n(X)\to qH_n(X);\]
and the homomorphisms $qH^n(P_\alpha)\to qH^n(X)$ induced by the maps $p^\infty_\alpha\:X\to P_\alpha$ factor through 
a unique homomorphism \[\tau^n_X\:pH^n(X)\to qH^n(X).\]

The $\bigsqcup$- and $\bigvee$-additivity properties yield that $\tau_n^X$ and $\tau^n_X$ are isomorphisms 
for each $n$ if $X$ is a disjoint union of compacta or a metric wedge of polyhedra, or more generally 
the result of an iterated application of $\bigsqcup$ and $\bigvee$ to compacta and polyhedra.

\begin{proposition} \label{0-cohomology}
$\tau^0_X\:pH^0(X)\to qH^0(X)$ is an isomorphism for all $X$.
\end{proposition}

\begin{proof} For a polyhedron $P$ we have $qH^0(P)\simeq H^0(P)$ by the Milnor short exact sequence.
Hence $pH^0(X)\simeq H^0(X)$.
Now $H^0(X)\simeq\Z^X$, where $\Z^X$ denotes the group of all continuous maps $X\to\Z$ (with respect to 
the discrete topology on $\Z$) with pointwise addition (see \cite{M00}*{Theorem \ref{book:0-homology}(a)}).
On the other hand $qH^0(X)\simeq\lim H^0(K_\alpha)\simeq\lim\Z^{K_\alpha}$ and $\tau^0_X$ is easily seen
to correspond to the homomorphism $\Z^X\to\lim\Z^{K_\alpha}$ yielded by the restriction maps
$\Z^X\to\Z^{K_\alpha}$.
But this homomorphism is a bijection (see \cite{M00}*{Example \ref{book:sit-example2}}).
\end{proof}

\begin{theorem} \label{nobeling-los}
Let $X$ be a local compactum. 
Then $\tau_0^X\:pH_0(X)\to qH_0(X)$ is surjective.
\end{theorem}

The proof uses two powerful results on abelian groups, due to N\"obeling and \L o\'s.

\begin{proof} Since $X$ is separable, we have $qH_0(X)\simeq\Hom(\Z^X,\Z)$, where $\Z^X$ denotes 
the group of all continuous maps $X\to\Z$ (with respect to the discrete topology on $\Z$) with 
pointwise addition (see \cite{M00}*{Theorem \ref{book:0-homology}(b)}).
Since $X$ is a local compactum, there exist compact subsets $K_1\subset K_2\subset\dots$ of $X$ 
such that $X=\bigcup_i K_i$ and each $K_i\subset\Int K_{i+1}$.
Then $\Z^X\simeq\lim\Z^{K_i}$ (see \cite{M00}*{Example \ref{book:sit-example2}}).
On the other hand, $pH_0(X)\simeq\colim qH_0(K_i)\simeq\colim\Hom(\Z^{K_i},\Z)$,
and $\tau_0^X$ is easily seen to correspond to the natural homomorphism
$\phi\:\colim\Hom(\Z^{K_i},\Z)\to\Hom(\lim\Z^{K_i},\Z)$.

If $S$ is a set and $A$ is an abelian group, let $FV(S,A)$ denote the group of all 
finite-valued functions $S\to A$ (with pointwise addition).
A subgroup $G$ of $FV(S,A)$ is called a {\it Specker subgroup} if for each $g\in G$ 
and each $a\in A\but\{0\}$ the indicator function $\I_{g^{-1}(a)}$ belongs to $G$.
Let us note that $\Z^{K_i}$ is a Specker subgroup of $FV(K_i,\Z)$.
(Indeed, if $g\:K_i\to\Z$ is continuous, then each $g^{-1}(a)$ is clopen, and 
consequently $\I_{g^{-1}(a)}$ is also continuous.)
Moreover, the image $G_i$ of $\Z^X$ in $\Z^{K_i}$ is also a Specker subgroup of $FV(K_i,\Z)$.
(Indeed, if $g=f|_{K_i}$, where $f\:X\to\Z$ is continuous, then each $f^{-1}(m)$ is clopen, 
hence $\I_{f^{-1}(m)}$ is continuous, and therefore so is 
$\I_{g^{-1}(m)}=\I_{f^{-1}(m)}|_{K_i}$.)
Then by N\"obeling's theorem (see \cite{Fu}*{Corollary 97.4 and Theorem 97.3}) $\Z^{K_i}$
is a free abelian group and $G_i$ is its direct summand.

From the above we get that $\Z^X\simeq\lim\big(\dots\to G_2\to G_1\big)$, where each $G_i$
is free abelian.
Then each $G_i$ is a projective $\Z$-module, so each $G_{i+1}\simeq G_i\oplus F_{i+1}$ for some
free abelian group $F_{i+1}$.
Writing $F_1=G_1$, we get that each $G_i\simeq F_1\oplus\dots\oplus F_i$ and 
$\Z^X\simeq\prod_{i=1}^\infty F_i$.
Then by \L o\'s's theorem (see \cite{Fu}*{Corollary 94.5}) every homomorphism 
$\Z^X\to\Z$ factors through some $G_i$.
Since $G_i$ is a direct summand in $\Z^{K_i}$, we further obtain that
every homomorphism $\Z^X\to\Z$ factors through some $\Z^{K_i}$.
In other words, $\phi$ is surjective.
\end{proof}

The homomorphism $\tau_n^X$ was studied already by P. S. Alexandroff in 1947, who constructed a $1$-dimensional 
Polish space $X$ with $qH_0(X,pt)=0$ and $pH_0(X,pt)\ne 0$ \cite{Al2}*{p.\ 214}.
On the other hand, answering a question raised independently by Alexandroff and S. Kaplan, a few years later
E. F. Mishchenko constructed a $2$-dimensional Polish space $Y$ with $pH_1(Y)=0$ and 
$qH_1(Y)\ne 0$ \cite{Mi} (see also \cite[p.\ 444, Remark 24]{Li2}).

\begin{remark} \label{jones}
There exists a Polish space $X$ which is connected (hence $qH_0(X,pt)=0$) 
but contains no nondegenerate compact connected subset (hence $pH_0(X,pt)\ne 0$, for if 
a pair of distinct points of $X$ lie in the same connected component of some compact subset 
of $X$, then they lie in a compact connected subset of $X$, which is a contradiction).
For instance, the graph of $f(x)=\sum_{n=1}^\infty 2^{-n}\sin(\frac1x-r_n)$,
where $r_n$ is the $n$th rational number (in some order) and $\sin(\infty)=0$
(see \cite[vol.\ II, \S47.IX]{Kur}).
In fact, such a space can even be locally connected \cite{Jo}.
\end{remark}

\subsection{Examples: homology}

The literature does not seem to contain an example of a local compactum $X$ whose $\tau_n^X$ is not an isomorphism.%
\footnote{It is erroneously claimed in \cite{Sk2}*{second page of the introduction} that $\N^+\x\N$ is 
such an example, where $\N^+$ is the one-point compactification of the countable discrete space $\N$.
This claim contradicts the fact that $\tau_n^X$ is an isomorphism if $X$ is a disjoint union of compacta.}
It is not hard to construct such an example, for which $\tau_n^X$ fails to be injective.

\begin{figure}[h]
\includegraphics[width=15.7cm]{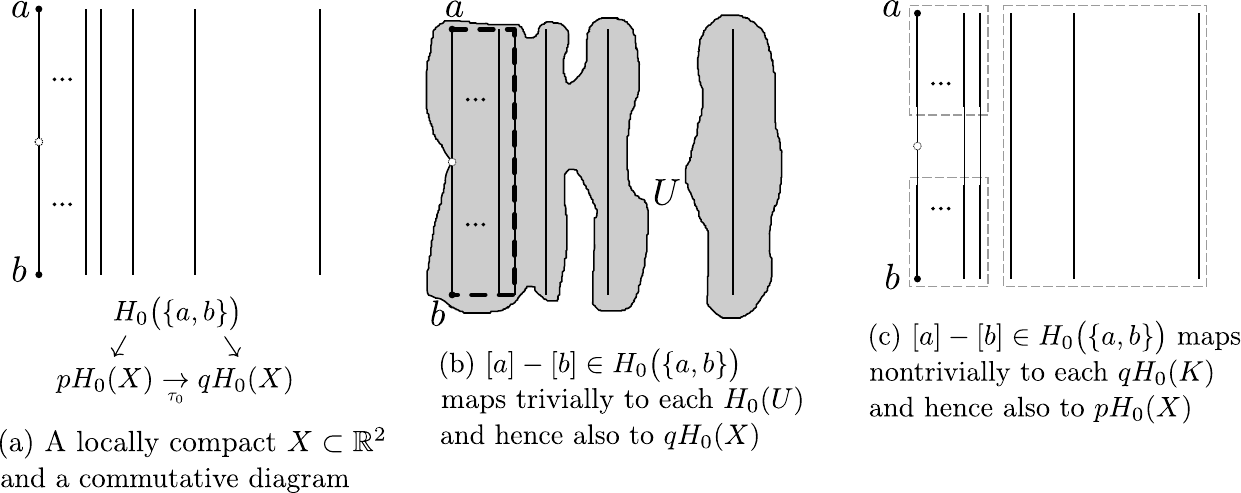}
\caption{$\tau_0$ need not be injective}
\label{alexandroff}
\end{figure}

\begin{example}\label{10.5} (See Figure \ref{alexandroff}.) 
Let $\frac1{\N^+}$ denote the subset $\{\frac1n\mid n\in\N\}\cup\{0\}$ 
of $[0,1]$, and let $X$ be the subspace $\frac1{\N^+}\times [-1,1]\but \{(0,0)\}$ of $\R^2$.
Clearly, $X$ is a local compactum.
To prove that $\tau_0^X$ is not injective, let $a=(0,1)$ and $b=(0,-1)$ and 
let $[a]-[b]\in H_0(\{a,b\})$ stand for a generator of $\ker\big(H_0(\{a,b\})\to H_0(\R^2)\big)$.
It suffices to show that $H_0(\{a,b\})\to qH_0(X)$ sends $[a]-[b]$ to zero, whereas $H_0(\{a,b\})\to pH_0(X)$ 
sends $[a]-[b]$ to a nontrivial element.

Indeed, if $U$ is an open neighborhood of $X$ in $\R^2$, then there exists an $n\in\N$ such that $U$ 
contains $[0,\frac1n]\x\{1,-1\}$.
Since $\{\frac1n\}\x[-1,1]$ as a subset of $X$ is also contained $U$, the image of $[a]-[b]$ in $H_0(U)$ is zero.
Hence $[a]-[b]$ maps trivially to $qH_0(X)$.

On the other hand, if $K\subset X$ is compact, then it is disjoint from $[0,\frac1k]\x[-\frac1k,\frac1k]$ 
for some $k\in\N$.
Hence $K$ lies in the union of three pairwise disjoint rectangles, $R_{10}=[\frac1{k-1},1]\x[-1,1]$, 
$R_{01}=[0,\frac1k]\x[\frac1k,1]$ and $R_{0,-1}=[0,\frac1k]\x[-1,-\frac1k]$; whereas $a$ and $b$ land 
in different $R_{ij}$'s.
Thus $[a]-[b]$ maps nontrivially to $qH_0(K)$ for each compact $K\subset X$ containing $\{a,b\}$, and therefore 
also nontrivially to $pH_0(X)$.
\end{example}

\begin{figure}[h]
\includegraphics[width=5cm]{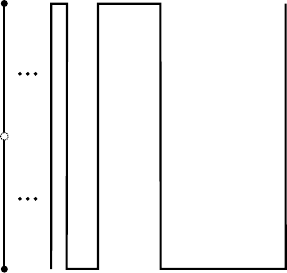}
\caption{A local compactum $X\subset\mathbb R^2$ with $qH_0(X)=0$ and $pH_0(X)\ne 0$}
\label{alexandroff'}
\end{figure}

\begin{remark} There is also a connected version of the previous example (see Figure \ref{alexandroff'}).
\end{remark}

\begin{remark}
Let $X'=X\cup r(X)$, where $r$ is the reflection $(x,y)\mapsto(-x,y)$ in $\R^2$.
It is clear from the above that $qH_1(K)\to H_1\big(\R^2\but\{(0,0)\}\big)$ is trivial for each compact 
$K\subset X'$, and therefore $pH_1(X')\to H_1\big(\R^2\but\{(0,0)\}\big)$ is trivial.
It might also seem that $qH_1(X')\to H_1\big(\R^2\but\{(0,0)\}\big)$ must be nontrivial,
which would imply that $\tau_1^{X'}$ is not surjective; but actually this is not the case.
Indeed, let \[U_n=\R^2\but\big(\{(0,0)\}\cup\{(\pm\tfrac1i,0)\mid i=n+\tfrac12,n+\tfrac32,\dots\}\cup
\{\pm\tfrac1i\mid i=\tfrac32,\tfrac52,\dots,n-\tfrac12\}\x\R\big).\]
Each $U_i$ is an open neighborhood of $X'$ in $\R^2\but\{(0,0)\}$, and $U_0\subset U_1\subset\dots$.
Hence $qH_1(X')\to H_1\big(\R^2\but\{(0,0)\}\big)$ factors through $\lim_i H_1(U_i)$, but the latter is
isomorphic to the inverse limit of the inclusions $\dots\to\bigoplus_{i=2}^\infty\Z\to\bigoplus_{i=1}^\infty\Z$,
which is zero.
\end{remark}

\begin{figure*}
\includegraphics[width=15.7cm]{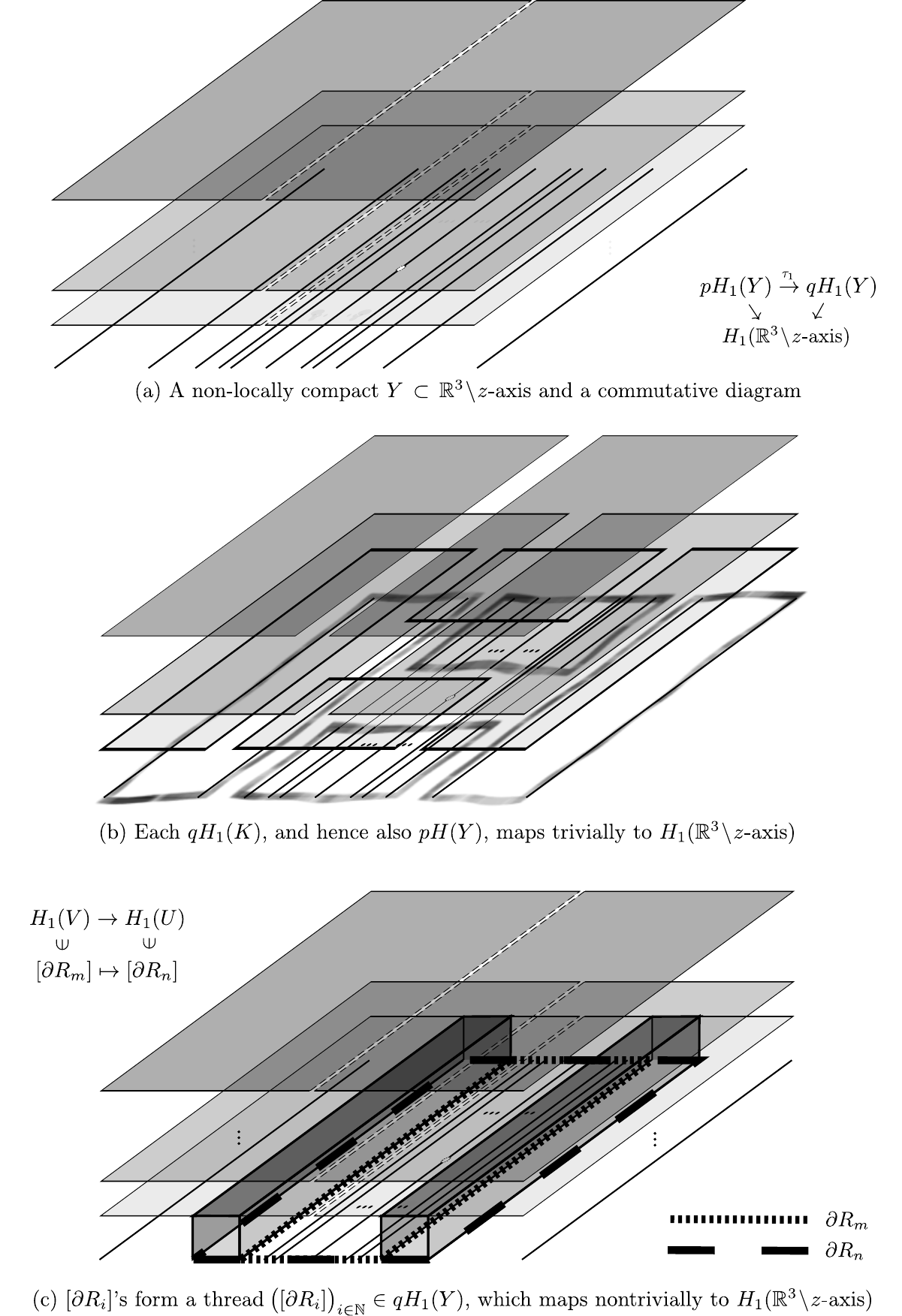}
\caption{$\tau_1$ need not be surjective}
\label{mishchenko}
\end{figure*}

\begin{example} \label{10.5'} (See Figure \ref{mishchenko}.)
Let $X'$ be as above and let $Y=X'\x\{0\}\cup\big([-1,0)\cup(0,1]\big)\x[-1,1]\x\frac1{\N}$,
where $\frac1{\N}=\{\frac1n\mid n\in\N\}$.
Let us note that $Y$ is not locally compact.
To prove that $\tau_1^Y$ is not surjective it suffices to show that 
$qH_1(K)\to H_1(\R^3\but z\text{-axis})$ is trivial for each compact $K\subset Y$, but 
$qH_1(Y)\to H_1(\R^3\but z\text{-axis})$ is nontrivial (where the ``$z$-axis'' is the line $\{(0,0)\}\x\R$).

Indeed, if $K\subset Y$ is compact, then it is disjoint
\begin{itemize}
\item from $[-\frac1k,\frac1k]\x[-\frac1k,\frac1k]\x[0,\frac1k]$ 
for some $k\in\N$;
\item from $\big([-\frac1{k+1/3},-\frac1{k+2/3}]\cup[\frac1{k+2/3},\frac1{k+1/3}]\big)\x[-1,1]\x[0,\frac1l]$ 
for some $l\in\N$, $l\ge k$; 
\item from $[-\eps,\eps]\x[-1,1]\x[\frac1{l-1},1]$ for some $\eps>0$.
\end{itemize} 
Then $K$ is contained in the union of six pairwise disjoint boxes (i.e.\ rectangular parallelepipeds),
\[\begin{aligned}
Q_{001}&=[\eps,1]\x[-1,1]\x[\tfrac1{l-1},1],
&Q_{0,0,-1}&=[-1,-\eps]\x[-1,1]\x[\tfrac1{l-1},1],\\
Q_{100}&=[\tfrac1{k+1/3},1]\x[-1,1]\x[0,\tfrac1l],
&Q_{-1,0,0}&=[-1,-\tfrac1{k+1/3}]\x[-1,1]\x[0,\tfrac1l],\\
Q_{010}&=[-\tfrac1{k+2/3},\tfrac1{k+2/3}]\x[\tfrac1k,1]\x[0,\tfrac1l],
&Q_{0,-1,0}&=[-\tfrac1{k+2/3},\tfrac1{k+2/3}]\x[-1,-\tfrac1k]\x[0,\tfrac1l],\\
\end{aligned}\]
none of which meets the $z$-axis.
Hence $qH_1(K)\to H_1(\R^3\but z\text{-axis})$ is trivial.
 
If $U$ is an open neighborhood of $Y$ in $\R^3$, then $U$ contains 
$[-\frac1n,\frac1n]\x\{1,-1\}\x\{0\}$ for some $n\in\N$.
Since $Y$ contains $\{-\frac1n,\frac1n\}\x[-1,1]\x\{0\}$, the entire boundary of the rectangle 
$R_n\bydef [-\frac1n,\frac1n]\x[1,-1]\x\{0\}$ lies in $U$.
Its clockwise oriented fundamental class $[\partial R_n]\in H_1(\partial R_n)$ maps to the clockwise oriented 
generator $\beta$ of $H_1(\R^3\but z\text{-axis})$.
If $V\subset U$ is another open neighborhood of $Y$ in $\R^3$, it contains $\partial R_m$ for some $m\in\N$, 
$m>n$.
Then $U$ also contains $\partial R_m$, and $[\partial R_n]-[\partial R_m]=[\partial R_{mn}\x\{0\}]\in H_1(U)$, 
where $R_{mn}=\big([-\frac1n,-\frac1m]\cup[\frac1m,\frac1n]\big)\x[-1,1]$.
Since $U$ contains $\partial R_{mn}\x\{0\}$, it also contains $(\partial R_{mn})\x[0,\frac1k]$ for some $k\in\N$.
On the other hand, $Y$ contains the two rectangles $R_{mn}\x\{\frac1k\}$, and so the cycle 
$\partial R_{mn}\x\{0\}$ bounds the chain $R_{mn}\x\{\frac1k\}\cup(\partial R_{mn})\x[0,\frac1k]$ in $U$.
Hence the bonding map $H_1(V)\to H_1(U)$ sends $[\partial R_m]$ to $[\partial R_n]$.
Thus the classes $[\partial R_n]\in H_1(U)$, $n=n(U)$, form a thread $x\in qH_1(Y)$, which maps
onto a generator $\beta\in H_1(\R^3\but z\text{-axis})$.
\end{example}

\begin{remark} In the previous example, the homomorphism $H_1(Y)\to qH_1(Y)$ is not surjective.
(This will follow from Corollary \ref{epi-epi}(a), but we can also see it directly.)
Indeed, we have $H_1(Y)=\colim H_1(K)$ over all compact $K\subset X$, and the same argument 
as above works to show that $H_1(K)\to H_1(\R^3\but z\text{-axis})$ is trivial for every 
compact $K\subset X$.
Hence $H_1(X)\to H_1(\R^3\but z\text{-axis})$ is trivial, and the assertion follows.
\end{remark}

\subsection{Examples: cohomology}

The previous examples along with the duality imply that $\tau^1_{S^2\but X}$ is non-injective and 
$\tau^1_{S^3\but Y}$ is non-surjective, where $S^n$ stands for the one-point compactification of $\R^n$.
Here are more explicit examples of this type, which also have lower dimensions, one additionally being 
local compacta.

\begin{example} \label{10.5''} The map $\tau^1_{X'}$ is non-injective, where $X'=X\cup r(X)$ is as above, 
a local compactum.

Indeed, let $\beta'\in H^1(\R^2\but\{(0,0)\})$ be a generator.
If $K\subset X'$ is compact, then it is disjoint from $[-\frac1k,\frac1k]\x[-\frac1k,\frac1k]$ for some $k\in\N$,
and consequently lies in the union of four pairwise disjoint rectangles, $R_{10}$, $r(R_{10})$, 
$R_{01}\cup r(R_{01})$ and $R_{0,-1}\cup r(R_{0,-1})$, none of which contains the origin $(0,0)$.
Hence $\beta'$ maps trivially to $H^1(K)$, and consequently also trivially to $qH^1(X')$.

On the other hand, if $U$ is an open neighborhood of $X'$ in $\R^2$, then there exists an $n\in\N$ such that $U$ 
contains $\partial R_n$. 
Since $\beta'$, or rather its image in $\Hom\big(H_1(\R^2\but\{(0,0)\}),\,\Z\big)$, evaluates nontrivially
on $[\partial R_n]$, the image of $\beta'$ in $H^1(K)$ for any compact $K\subset U$ containing $\partial R_n$
also evaluates nontrivially on $[\partial R_n]$.
Hence $\beta'$ maps nontrivially to $qH^1(U)$, and consequently also nontrivially to $pH^1(X')$.
\end{example}

\begin{example} \label{10.5'''} Let 
$Z=X\x\{0\}\cup\{0\}\x[-1,1]\x\big([-1,0)\cup(0,1]\big)\cup[0,1]\x[-1,1]\x\frac1{\pm\N}$,
where $\frac1{\pm\N}=\{\frac1{\pm n}\mid n\in\N\}$.
Let us note that $Z$ is not locally compact.
To prove that $\tau^1_Z$ is not surjective, let us note that $Z$ contains the boundary of the rectangle 
$R\bydef \{0\}\x[-1,1]\x[-1,1]$, and let $\gamma\in H^1(\partial R)$ be a generator.
It suffices to show that $\gamma$ lies in the image of $qH^1(Z)\to H^1(\partial R)$, but not in
that of $pH^1(Z)\to H^1(\partial R)$.

If $U$ is an open neighborhood of $Z$ in $\R^3$, then $U$ contains $[0,\frac1n]\x\{1,-1\}\x[-1,1]$ 
for some $n\in\N$ as well as $\{\frac1n\}\x[-1,1]\x[-\frac1m,\frac1m]$ for some $m\in\N$. 
Then $U$ contains the intersection $V_{mn}$ of $[0,\infty)\x\R^2$ with the boundary of the box 
$[-\frac1n,\frac1n]\x[-1,1]\x[-\frac1m,\frac1m]$.
Also $U$ contains the two rectangles $W_m\bydef \{0\}\x[-1,1]\x\big([-1,-\frac1m]\cup[\frac1m,1]\big)$.
It follows that $\partial R$ is null-homotopic in $U$, and more specifically in $V_{mn}\cup W_m$.
Since $\gamma$ evaluates nontrivially on $[\partial R]\in H_1(\partial R)$, it follows that $\gamma$ is not
in the image of $H^1(V_{mn}\cup W_m)$.
Therefore it is not in the image of $qH^1(U)$, and consequently also not in the image of $pH^1(Z)$.

If $K\subset Z$ is compact, then it is disjoint from $[0,\frac1k]\x[-\frac1k,\frac1k]\x[-\frac1k,\frac1k]$ 
for some $k\in\N$, as well as from $[\frac1{k+2/3},\frac1{k+1/3}]\x[-1,1]\x[-\frac1l,\frac1l]$ for some 
$l\in\N$, $l\ge k$.
Hence $K$ lies in the union of two disjoint polyhedra: the box 
$B_{kl}=[\tfrac1{k+1/3},1]\x[-1,1]\x[-\frac1l,\tfrac1l]$ and the union $C_{kl}$ of the $yz$ plane $\{0\}\x\R^2$
with four pairwise disjoint boxes
\[\begin{aligned}
Q_{01}&=[0,1]\x[-1,1]\x[\tfrac1{l-1},1],
&Q_{0,-1}&=[0,1]\x[-1,1]\x[-1,-\tfrac1{l-1}],\\
Q_{10}&=[0,\tfrac1{k+2/3}]\x[\tfrac1k,1]\x[-\tfrac1l,\tfrac1l],
&Q_{-1,0}&=[0,\tfrac1{k+2/3}]\x[-1,-\tfrac1k]\x[-\tfrac1l,\tfrac1l].\\
\end{aligned}\]
Let $\rho\:\R^3\but\R\x\{(0,0)\}=\R\x\mathbb C\but\R\x\{0\}\to S^1$ be defined by 
$(x,\rho e^{i\phi})\mapsto e^{i\phi}$, and let $c\:\R^3\to S^1$ be the constant map $(x,y,z)\mapsto 1$.
We define a map $f_{kl}\:B_{kl}\cup C_{kl}\to\partial R$ by $f_{kl}=c|_{B_{kl}}\cup\rho|_{C_{kl}}$.
If $K'\supset K$ is another compact subset of $Z$, then $K'$ lies in $B_{k'l'}\cup C_{k'l'}$ for some $k'$ and $l'$.
Let us note that $f_{kl}$ and $f_{k'l'}$ agree on $B_{kl}\cap B_{k'l'}$ and on $C_{kl}\cap C_{k'l'}$.
Since $B_{k'l'}$ is disjoint from the $yz$ plane, $A\bydef B_{k'l'}\cap C_{kl}$ lies in the union of the four 
pairwise disjoint boxes $B_{k'l'}\cap Q_{ij}$.
Hence $f_{kl}|_A$ is null-homotopic, i.e.\ homotopic to $c|_A=f_{k'l'}|_A$.
Similarly, if $A'=B_{kl}\cap C_{k'l'}$, then $f_{k'l'}|_{A'}$ is homotopic to $c|_{A'}=f_{kl}|_{A'}$.
Consequently $f_{kl}|_K$ and $f_{k'l'}|_K$ are homotopic.
Thus the homotopy classes $[f_{kl}|_K]\in[K,S^1]\simeq H^1(K)$ form a thread $\delta\in qH^1(U)$,
which clearly maps onto a generator $\gamma\in H^1(\partial R)$.
\end{example}  

\begin{remark} In the previous example, the homomorphism $H^1(Z)\to qH^1(Z)$ is not surjective.
(This will follow from Corollary \ref{epi-epi}(b), but we can also see it directly.)

Indeed, suppose that $\gamma\in H^1(\partial R)$ lies in the image of $H^1(Z)$.
Since $H^1(Z)\simeq [Z,S^1]$, there exists a map $f\:Z\to S^1$ such that $f|_{\partial R}$ is 
not null-homotopic.
Let $M=Z\cup[0,1]\x\{-1,1\}\x[-1,1]\cup\frac1{\N}\x[-1,1]\x[-1,1]$.
Clearly $Z$ is closed in $M$.
Since $S^1$ is an ANR, $f$ extends to a map $\bar f\:U\to S^1$ for some open neighborhood $U$ 
of $Z$ in $M$.
But clearly $U$ contains $V_{mn}\cup W_m$ for some $m,n\in\N$.
Hence $\partial R$ is null-homotopic in $U$, so we get a contradiction.
\end{remark}

If $X$ is as in Example \ref{10.5} and $X'$ is as in Example \ref{10.5''}, then $S^2\but X$ 
and $S^2\but X'$ are coronated polyhedra (see \cite{M-V}*{Example \ref{cor:co-locally-compact}}) 
such that $\tau^1_{S^2\but X}$ and $\tau_0^{S^2\but X'}$ are not injective.
It is not hard to construct one-dimensional examples of this type:

\begin{example} \label{compactohedral example}
Let $X=\{\frac1n\mid n\in\N\}\x[-1,1]\cup\{0\}\x\{-1,1\}\subset\R^2$.
Also let $r$ be the reflection $(x,y)\mapsto(-x,y)$ of $\R^2$, and let $X'=X\cup r(X)$.
Clearly, $X$ and $X'$ are non-locally-compact coronated polyhedra.
Similarly to Example \ref{10.5}, $\tau_0^X\:pH_0(X)\to qH_0(X)$ is not injective; 
and similarly to Example \ref{10.5''}, $\tau^1_{X'}\:pH^1(X')\to qH^1(X')$ is not injective.
\end{example}

\subsection{Surjectivity of $\tau$}\label{surjectivity}

The universal property of direct limits implies that the inclusion induced homomorphisms
$q^kH_n(K)\to q^kH_n(X)$, where $K$ runs over compact subsets of $X$, factor through
a unique homomorphism \[\tau_n^{(k)}\:p^kH_n(X)\to q^kH_n(X);\] and the homomorphisms
$q^kH^n(P_\alpha)\to q^kH^n(X)$ induced by the maps $p^\infty_\alpha\:X\to P_\alpha$ factor through
a unique homomorphism \[\tau^n_{(k)}\:p^kH^n(X)\to q^kH^n(X).\]

\begin{theorem}\label{10.11} (a) If $X$ is a local compactum, then
there is a short exact sequence 
\[0\xr{\phantom{\tau^n}} p^1H^{n-1}(X)\xr{\tau^{n-1}_{(1)}}q^1H^{n-1}(X)\xr{\phantom{\tau}} 
pH^n(X)\xr{\tau^n}qH^n(X)\xr{\phantom{\tau^n}} 0.\]

(b) If $X$ is a coronated polyhedron, then there is a short exact sequence 
\[0\xr{\phantom{\tau_n}} p^1H_{n+1}(X)\xr{\tau_{n+1}^{(1)}}q^1H_{n+1}(X)\xr{\phantom{\tau}} 
pH_n(X)\xr{\tau_n}qH_n(X)\xr{\phantom{\tau_n}} 0.\]
\end{theorem}

Let us recall that $\tau^1_X$ fails to be surjective for a certain separable metrizable $X$ 
(Example \ref{10.5'''})
and fails to be injective for a certain local compactum (Example \ref{10.5''}).
Also, $\tau_1^X$ fails to be surjective for a certain separable metrizable $X$ (Example \ref{10.5'})
and $\tau_0^X$ fails to be injective for a certain coronated polyhedron $X$ 
(Example \ref{compactohedral example}).

\begin{proof}[Proof. (a)] Since star-finite covers are cofinal in the directed set of open covers of $X$, 
the naturality in Theorem \ref{10.10'} implies that its short exact sequence combines with that of 
Corollary \ref{milnor-ses''} into a commutative diagram
$$\begin{CD}
0@>>>p^1H^{n-1}(X)@>>>H^n(X)@>>>pH^n(X)@>>>0\\
@.@V\tau^{n-1}_{(1)}VV@|@V\tau^nVV@.\\
0@>>>q^1H^{n-1}(X)@>>>H^n(X)@>>>qH^n(X)@>>>0
\end{CD}$$
This diagram implies that $\tau^n$ is surjective, $\tau^{n-1}_{(1)}$ is injective, and 
(using the third isomorphism theorem $(G/K)/(H/K)\simeq G/H$) that
$\ker\tau^n\simeq\coker\tau^{n-1}_{(1)}$.
\end{proof}

\begin{proof}[Proof. (b)] This is similar, using the naturality in Theorem \ref{co10.10'}.
\end{proof}

\begin{corollary}
If $X$ is a $d$-dimensional coronated polyhedron, then $\tau_d^X$ is an isomorphism.%
\footnote{We may similarly conclude that if $X$ is a local compactum, 
then $\tau^0_X$ is an isomorphism, but a stronger assertion was already proved more easily 
in Proposition \ref{0-cohomology}.}
\end{corollary}

For homology of local compacta, for cohomology of coronated polyhedra and for spaces that are neither
coronated polyhedra nor local compacta, the proof of Theorem \ref{10.11} yields only the following

\begin{corollary} \label{epi-epi} Let $X$ be a metrizable space.

(a) $pH_n(X)\xr{\tau_n^X}qH_n(X)$ is surjective if and only if so is 
$H_n(X)\xr{\gamma_n^X}qH_n(X)$.

(b) $pH^n(X)\xr{\tau^n_X}qH^n(X)$ is surjective if and only if so is 
$H^n(X)\xr{\gamma^n_X}qH^n(X)$.
\end{corollary}

\begin{remark} $\gamma_n^X$ and $\gamma^n_X$ are special cases of homomorphisms
$H_n(\lim X_\alpha)\xr{\gamma}\lim H_n(X_\alpha)$ and 
$H^n(\bigcup_\alpha X_\alpha)\xr{\gamma}\lim H^n(X_\alpha)$.
These were widely studied and some conclusive results have been obtained, albeit in situations that are 
far from what we pursue here.

If $K_\alpha$ are compact Hausdorff spaces%
\footnote{The authors of \cite{KS} and \cite{Md2} do not state 
the compactness hypothesis, but it is needed to ensure that their homology theories $H_*^{KS}$, $H_*^{Md}$ 
are isomorphic to the standard one (in the sense of the uniqueness theorems, see 
\cite{M-II}*{\S\ref{axi:general-uniqueness}}).
Indeed, their theories both satisfy the ``dual'' universal coefficient formula with respect to \v Cech 
cohomology, $0\to\Ext\big(H^{n+1}(X),\,G\big)\to H_n^{KS/Md}(X;\,G)\to\Hom\big(H^n(X),\,G\big)\to 0$, and 
therefore are non-standard (i.e.\ do not satisfy $\bigsqcup$-additivity) already on non-compact polyhedra, 
e.g.\ $H_{-1}^{KS}(\N)\simeq H_{-1}^{Md}(\N)\simeq\Ext(\prod_{\N}\Z,\,\Z)\ne 0$, where $\N$ is the countable 
discrete space (concerning $\Ext(\prod_{\N}\Z,\,\Z)$, see references in \cite{M00}*{Remark \ref{book:UCF}(2)}), 
and also without using any deep results, 
$H_2^{KS}(\N\x\R P^2;\,\Z/2)\simeq H_2^{Md}(\N\x\R P^2;\,\Z/2)\simeq\Hom(\prod_{\N}\Z/2,\,\Z/2)$, which 
is uncountable (since $\prod_{\N}\Z/2$ is a vector space over $\Z/2$) and so is not isomorphic to 
$\bigoplus_{\N}\Z/2$.} 
indexed by a directed set and $K=\lim K_\alpha$, then there is a long exact sequence of 
Mdzinarishvili \cite{Md}*{Bespiel 3} (see also \cite{Md2}*{Example 5})
\[\scalebox{0.8}{$
\dots\to\invlim^3 H_{n+2}(K_\alpha)\to\invlim^1 H_{n+1}(K_\alpha)\to H_n(K)\xr{\gamma}
\invlim H_n(K_\alpha)\to\invlim^2 H_{n+1}(K_\alpha)\to\invlim^4 H_{n+2}(K_\alpha)\to\dots
$}\]
and an exact sequence of Kuz'minov--Shvedov \cite{KS}*{Theorem 5(1)}%
\[\scalebox{0.95}{$
0\to\invlim^1\Hom\big(H^{n+1}(K_\alpha),\,\Z\big)\to H_n(K)\xr{\gamma}
\invlim H_n(K_\alpha)\to\invlim^2\Hom\big(H^{n+1}(K_\alpha),\,\Z\big)\to 0.
$}\]
If either the indexing set is countable (and therefore has a cofinal sequence) or each $K_\alpha$ is 
a compact polyhedron (and therefore has finitely generated homology), then 
$\invlim^p H_i(K_\alpha)=0$ for $p\ge 2$ and we recover a Milnor-type short exact sequence
(see Remark \ref{lim2fg} above and \cite{M00}*{Remark \ref{book:lim2seq}}).

If $U_\alpha$ are open subsets of $X\bydef \bigcup_\alpha U_\alpha$ (and there are no restrictions on 
the space $X$), then one could expect a Mdzinarishvili-type exact sequence in cohomology 
(see \cite{Sk5}*{\S\S8.1.1--8.1.2}); at any rate, there is a spectral sequence of the form 
$E_2^{pq}=\invlim^p H^q(U_\alpha)\Rightarrow H^{p+q}(X)$ which implies that
$\gamma\:H^0(X)\to\invlim H^0(U_\alpha)$ is an isomorphism and there is an exact sequence 
\[
0\to\invlim^1 H^0(U_\alpha)\to H^1(X)\xr{\gamma}
\invlim H^1(U_\alpha)\to\invlim^2 H^0(U_\alpha)\to H^2(\invlim U_\alpha).
\]
\cite{Sk5}*{\S8.1.5}, \cite{Sk92}.
One special case of this situation occurs when $X$ is a simplicial complex (with metric 
or weak topology) and each $U_\alpha$ is the union of the open stars of all vertices of some 
subcomplex $P_\alpha$.
In this case, the above spectral sequence is due to Bousfield--Kan \cite{BK} and Araki--Yosimura \cite{AY},
and there is indeed a (finite) Mdzinarishvili-type long exact sequence \cite{Oh}
\[\scalebox{0.8}{$
\dots\to\invlim^3 H^{n-2}(P_\alpha)\to\invlim^1 H^{n-1}(P_\alpha)\to H^n(X)\xr{\gamma}
\invlim H^n(P_\alpha)\to\invlim^2 H^{n-1}(P_\alpha)\to\invlim^4 H^{n-2}(P_\alpha)\to\dots.
$}\]
If either the indexing set is countable (and therefore has a cofinal sequence), or the subcomplexes
$P_\alpha$ are finite (and so have finitely generated cohomology), then 
$\invlim^p H^i(P_\alpha)=0$ for $p\ge 2$ and we again recover a Milnor-type short exact sequence
(see Remark \ref{lim2fg} above and \cite{M00}*{Remark \ref{book:lim2seq}}).

It should also be noted that Mdzinarishvili himself \cite{Md}*{Bespiel 1}, \cite{Md3} established 
a (finite) long exact sequence
\[\scalebox{0.8}{$
\dots\to\invlim^3 H^{n-2}_s(K_\alpha)\to\invlim^1 H^{n-1}_s(K_\alpha)\to H^n_s(X)\xr{\gamma}
\invlim H^n_s(K_\alpha)\to\invlim^2 H^{n-1}_s(K_\alpha)\to\invlim^4 H^{n-2}_s(K_\alpha)\to\dots
$}\]
where $X$ is an arbitrary space, $K_\alpha$ are its compact subsets and $H^*_s$ is singular cohomology.
\end{remark}

\subsection{Finitely generated derived limit}
Given an inverse sequence $\dots\xr{f_1} G_1\xr{f_0} G_0$ of groups, 
we may think of it as the colimit of the direct system of all inverse 
sequences $\dots\to H^\alpha_1\to H^\alpha_0$ such that each $H^\alpha_i$ is a finitely generated subgroup of 
$G_i$ and the bonding maps $H^\alpha_{i+1}\to H^\alpha_i$ are the restrictions of the $f_i$.
(Here the $G_i$ are possibly non-abelian, though the abelian case suffices for the purposes of the present paper.
We do not assume that the $G_i$ are countable, but this case would suffice for all current and expected 
applications.)
We define $\derlimfg G_i\bydef \colim_\alpha\derlim_i H^\alpha_i$ and 
$\invlimfg G_i\bydef \colim_\alpha\lim_i H^\alpha_i$.

\begin{proposition}\label{9.A} The natural homomorphism $\invlimfg G_i\to\lim G_i$ is an isomorphism, and
the natural map $\tau_\fg\:\derlimfg G_i\to\derlim G_i$ is injective.
\end{proposition}

\begin{proof} A thread of the inverse sequence $\dots\to G_1/H^\alpha_1\to G_0/H^\alpha_0$
of right cosets is a sequence of the form
$\gamma=(\dots,g_1H^\alpha_1,g_0H^\alpha_0)$.
Each $g_iH^\alpha_i$ lies in the finitely generated subgroup
$J^\alpha_i=\left<g_i,H^\alpha_i\right>$.
Each bonding map $G_{i+1}\to G_i$ sends $g_{i+1}$ into $g_ih_i$ for some
$h_i\in H^\alpha_i$, and hence $J^\alpha_{i+1}$ into $J^\alpha_i$.
Thus the $J^\alpha_i$ are of the form $H^\beta_i$ for some
$\beta=\beta(\alpha)$, so the thread $\gamma$ maps trivially into
$\colim_\alpha\lim\nolimits_i G_i/H^\alpha_i$.
The assertion now follows from the exact sequence
\[0\to\lim\nolimits_i H^\alpha_i\to\lim\nolimits_i G_i\to\lim\nolimits_i G_i/H^\alpha_i\to
\derlim\nolimits_i H^\alpha_i\to\derlim\nolimits_i G_i\]
(see \cite[3.2(d$'$)]{M1}) by applying the exact functor $\colim_\alpha$.
\end{proof}

\begin{example} Clearly, $\derlimfg\big(\dots\xr{p}\Z\xr{p}\Z\big)=
\derlim\big(\dots\xr{p}\Z\xr{p}\Z\big)\simeq\hat\Z_p/\Z$.

On the other hand, $\derlim\Big(\dots\emb\bigoplus_{i=2}^\infty\Z\emb\bigoplus_{i=1}^\infty\Z\Big)\simeq
\prod_{i=1}^\infty\Z/\bigoplus_{i=1}^\infty\Z\ne 0$, but
$\derlimfg\Big(\dots\emb\bigoplus_{i=2}^\infty\Z\emb\bigoplus_{i=1}^\infty\Z\Big)=0$.
Indeed, let us consider more generally an inverse sequence $\dots\emb F_1\emb F_0$ of pure inclusions 
between free abelian groups.
(A subgroup $H$ of a group $G$ is called {\it pure} if for any $g\in G$ and $n\in\Z$ such that $ng\in H$
one has $g\in H$.)
For a finitely generated subgroup $H$ of a free abelian group $F$, the purification
$\bar H\bydef \{g\in F\mid \exists n\in\mathbb Z: ng\in H\}$ of $H$ is finitely generated
(by tensoring with $\Q$; compare \cite{Fu}*{\S26, remarks after (e) and after 26.2}).
Hence every subsequence $\dots\emb H_1\emb H_0$ of finitely generated subgroups in the inverse sequence
$\dots\emb F_1\emb F_0$ is included in the subsequence $\dots\emb\bar H_1\emb\bar H_0$ of purifications, 
which are finitely generated.
Since each $\bar H_{i+1}$ is pure in $F_{i+1}$, which in turn is pure in $F_i$, we get that 
$\bar H_{i+1}$ is pure in $F_i$ and hence also in $\bar H_i$.
But then clearly $\dots\emb\bar H_1\emb\bar H_0$ satisfies the Mittag--Leffler condition, and hence
$\derlim\big(\dots\emb\bar H_1\emb\bar H_0\big)=0$.
\end{example}

\begin{remark} In connection with the previous example, let us note that there exists an inverse sequence
of embeddings $\dots\to G_1\to G_0$ between free abelian groups of rank 2, and an inverse sequence
$\dots\to H_1\to H_0$ of their {\it pure} subgroups of rank $1$ such that the map $\derlim H_i\to\derlim G_i$
is not injective \cite{M00}*{Example \ref{book:lattice2}}.
\end{remark}

\subsection{Kernel of $\tau$: cohomology}

\begin{theorem}\label{10.13} Let $X$ be a local compactum,
$X=\bigcup K_i$, where each $K_i$ is compact, $K_i\subset\Int K_{i+1}$.
Then 

(a) $p^1H^{n-1}(X)\simeq\derlimfg H^{n-1}(K_i)$;

(b) there is a short exact sequence
\[0\xr{\phantom{\tau}}\derlim H^{n-1}(K_i)/\derlimfg H^{n-1}(K_i)\xr{\phantom{\tau}} 
pH^n(X)\xr{\tau^n_X}qH^n(X)\xr{\phantom{\tau}} 0.\]
\end{theorem}

\begin{proof} Let $G_i=H^{n-1}(K_i)$.
We will prove that the image of $\tau^{n-1}_{(1)}\:p^1H^{n-1}(X)\to q^1H^{n-1}(X)$ 
coincides with the subgroup $\derlimfg G_i$ of $\derlim G_i=q^1H^{n-1}(X)$.
This together with Theorem \ref{10.11} implies both (a) and (b).

For each locally compact separable polyhedron $P$ we have $q^1H^{n-1}(P)=\derlim H^{n-1}(L_i)$, where 
the $L_i$ are compact polyhedra with $\bigcup L_i=P$ and each $L_i\subset\Int L_{i+1}$.
But the groups $H^{n-1}(L_i)$ are finitely generated, so $\derlim H^{n-1}(L_i)=\derlimfg H^{n-1}(L_i)$.
Since star-finite covers are cofinal in the directed set of open covers of $X$ (this only needs $X$ to be
separable metrizable), it follows that the image of $\tau^{n-1}_{(1)}\:p^1H^{n-1}(X)\to q^1H^{n-1}(X)$ 
lies in $\derlimfg G_i$.

Conversely, since $X$ is a local compactum, it is the limit of an inverse sequence $\dots\xr{p_1} P_1\xr{p_0} P_0$
of locally compact separable polyhedra and proper maps (see \cite{M00}*{Theorem \ref{book:isbell}}).
Then there exist compact subpolyhedra $Q_{ij}\subset P_j$ such that each $Q_{ij}\subset\Int Q_{i+1,j}$, each 
$P_j=\bigcup_{i=1}^\infty Q_{ij}$, each $p_j^{-1}(Q_{ij})=Q_{i,j+1}$, and each $K_i=\lim_j Q_{ij}$.
(In fact, we may assume that $P_0=[0,\infty)$ and each $Q_{i0}=[0,i]$.)
Let $Q_{i,[0,\infty]}$ denote the compactified mapping telescope of $\dots\to Q_{i1}\to Q_{i0}$ and 
for a $k$-tuple $s=(s_1,\dots,s_k)\in\N^k$, let 
$Q_{k,[s,\infty]}=Q_{1,[s_1,\infty]}\cup Q_{2,[s_2,\infty]}\cup\dots\cup Q_{k,[s_k,\infty]}
\subset Q_{k,[0,\infty]}$.
(Let us note that $Q_{i,[s_i,\infty]}\cap Q_{i+1,[s_{i+1},\infty]}=Q_{i,[s_{i+1},\infty]}$ if $s_{i+1}\ge s_i$.
It would suffice to consider only $s$ such that each $s_{i+1}\ge s_i$, but we will do without this assumption
in order to simplify notation.)
Also let $P_{[0,\infty]}$ be the extended mapping telescope of $\dots\to P_1\to P_0$ and for a sequence 
$\sigma=(\sigma_i)_{i\in\N}\in\N^{\N}$ let $P_{[\sigma,\infty]}\subset P_{[0,\infty]}$ be the union of 
the increasing chain $Q_{1,[\sigma|_{[1]},\infty]}\subset Q_{2,[\sigma|_{[2]},\infty]}\subset\dots$, 
where $\sigma|_{[i]}$ denotes the $i$-tuple of the first $i$ members of $\sigma$.

Since $Q_{i,[0,\infty]}$ deformation retracts onto $Q_{i0}$, we have $H^*(Q_{i,[0,\infty]}),\,Q_{i0})=0$
and consequently $G_i=H^{n-1}(K_i)\simeq H^n(Q_{i,[0,\infty]},\,K_i\cup Q_{i0})$.
We have $G_i\simeq\colim_j H^{n-1}(Q_{i,[j,\infty]})$ and consequently 
$G_i\simeq\colim_{s\in\N^i} G_{is}$, where \[G_{is}=H^{n-1}(Q_{i,[s,\infty]})\simeq 
H^n(Q_{i,[0,\infty]},\,Q_{i,[s,\infty]}\cup Q_{i0})\] and the poset $\N^i$ is 
the product of $k$ copies of the poset $\N$ (in other words, $s\le t$ if $s_k\le t_k$ for each $k\le i$).
(Indeed, $\N^i$ has a cofinal sequence consisting of the constant $i$-tuples $n=(n,\dots,n)$.)
On the other hand, $p^1H^{n-1}(X)\simeq\colim_{\sigma\in\N^{\N}}\Gamma_\sigma$,
where \[\Gamma_\sigma=q^1H^{n-1}(P_{[\sigma,\infty]})\simeq\derlim G_{i\sigma|_{[i]}}\simeq
q^1H^n(P_{[0,\infty]},\,P_{[\sigma,\infty]}\cup P_0)\] 
and the poset $\N^{\N}$ is the product of countably many copies of the poset $\N$ (in other words, 
$\sigma\le\tau$ if $\sigma_k\le\tau_k$ for each $k\in\N$).
Let us note that the the groups $G_{is}\simeq H^n(Q_{i,[0,\,\max s]},\,Q_{i,[s,\,\max s]}\cup Q_{i0})$ 
are finitely generated.

Let $h_i\:G_{i+1}\to G_i$ be the restriction homomorphism.
To show that \[\tau^{n-1}_{(1)}\:\colim_{\sigma\in\N^{\N}}\derlim_i G_{i\sigma|_{[i]}}\to\derlim_i G_i\] 
is onto $\derlimfg G_i=\colim_F\derlim F_i$, where $F$ ranges over all collections
of finitely generated subgroups $F_i\subset G_i$, $i=1,2,\dots$, with
$h_i(F_{i+1})\subset F_i$ for each $i$, it suffices to show that every such
$F$ is contained in an $F'$ for which there exists a sequence $\sigma\in\N^{\N}$
such that $\prod_i F_i'$ equals the image of $\prod_i G_{i\sigma|_{[i]}}$ in $\prod_i G_i$.
Since each $G_{i\sigma|_{[i]}}$ is finitely generated, $F'_i$ may be {\it defined}
as its image in $G_i$, provided that it contains $F_i$.
Thus it remains to prove the following lemma. 

\begin{lemma}\label{10.14}
Let $F_i\subset G_i$, $i\in\N$, be finitely generated subgroups such that each $h_i(F_{i+1})\subset F_i$.
Then there exists a sequence $\sigma\in\N^{\N}$ such that each $F_i$, $i\in\N$, lies 
in the image of $G_{i\sigma|_{[i]}}$ in $G_i$.
\end{lemma}

\begin{proof}
Let $\alpha_1,\dots,\alpha_r$ be a set of generators of $F_1$.
Since $G_1\simeq\colim_{n\in\N} G_{1n}$, each $\alpha_i$ lies in the image of $G_{1n_i}$ for some $n_i$.
Consequently all of the $\alpha_i$, and hence also $F_1$, lie in the image of $G_{1s_1}$, where 
$s_1=\max(n_1,\dots,n_r)$.

Assuming that $s=(s_1,\dots,s_k)\in\N^k$ is a $k$-tuple such that $F_k$ lies in the image of $G_{ks}$ in $G_k$, 
we will construct an $s_{k+1}\in\N$ such that $F_{k+1}$ lies in the image of $G_{k+1,\,s*s_{k+1}}$ in $G_{k+1}$, 
where $s*s_{k+1}$ denotes the $(k+1)$-tuple $(s_1,\dots,s_{k+1})$.
The assertion of the lemma will then follow by an inductive application of this construction.

Let $\alpha_1,\dots,\alpha_r$ be a set of generators of $F_{k+1}$.
Then each $h_k(\alpha_i)\in F_k$, so by our hypothesis $h_k(\alpha_i)$ is the image of some 
$\beta_i\in G_{ks}$.
Thus each $\alpha_i$ lies in the image of the pullback of the diagram
\smallskip
\[\begin{CD}
@.G_{ks}\\
@.@VVV\\
G_{k+1}@>h_k>>G_k.
\end{CD}\]
\smallskip

\noindent
But the latter diagram is precisely the lower right corner of the commutative diagram
\bigskip
\[\begin{CD}
H^n(Q_{k+1,[0,\infty]},\,Q_{k,[s,\infty]}\cup K_{k+1}\cup Q_{k+1,0})@>>>
H^n(Q_{k,[0,\infty]},\,Q_{k,[s,\infty]}\cup Q_{k0})\\
@VVV@VVV\\
H^n(Q_{k+1,[0,\infty]},\,K_{k+1}\cup Q_{k+1,0})@>>>H^n(Q_{k,[0,\infty]},\,K_k\cup Q_{k0}),
\end{CD}\]
\smallskip

\noindent
which is a special case of the diagram in Lemma \ref{pullback-lemma}(a) below --- except that
$Q_{k+1,[0,\infty]}$ is generally not a polyhedron.
However, we can rewrite the previous diagram in terms of cohomology with compact support of
locally compact polyhedra:
\bigskip
\[\begin{CD}
H^n_c(Q_{k+1,[0,\infty)},\,Q_{k,[s,\infty)}\cup Q_{k+1,0})@>>>
H^n_c(Q_{k,[0,\infty)},\,Q_{k,[s,\infty)}\cup Q_{k0})\\
@VVV@VVV\\
H^n_c(Q_{k+1,[0,\infty)},\,Q_{k+1,0})@>>>H^n_c(Q_{k,[0,\infty)},\,Q_{k0}),
\end{CD}\]
\smallskip

\noindent
and the proof of Lemma \ref{pullback-lemma}(a) works for compactly supported cohomology.
Hence each $\alpha_i$ is the image of some
$\gamma_i\in H^n(Q_{k+1,[0,\infty]},\,Q_{k,[s,\infty]}\cup Q_{k+1,0})$.
But the latter group is isomorphic to $\colim_{n\in\N} G_{k+1,\,s*n}\simeq
\colim_{n\in\N}H^n(Q_{k+1,[0,\infty]},\,Q_{k,[s*n,\infty]}\cup Q_{k+1,0})$.
Hence each $\gamma_i$ lies in the image of $G_{k+1,\,s*n_i}$ for some $n_i$.
Consequently all of the $\gamma_i$ lie in the image of $G_{k+1,\,s*s_{k+1}}$, where 
$s_{k+1}=\max(n_1,\dots,n_r)$.
But then all of the $\alpha_i$, and hence also $F_{k+1}$, lie in the image of 
$G_{k+1,\,s*s_{k+1}}$ in $G_{k+1}$.
\end{proof}
\end{proof}
 
\begin{lemma} \label{pullback-lemma}
Let $X$ be a polyhedron and $Z\subset Y\subset X$ and $W\subset X$ its subpolyhedra, and let
$W'=W\cap Y$.
In the commutative diagrams
\smallskip
\[\begin{CD}
H^n(X,\,Z\cup W)@>>>H^n(Y,\,Z\cup W')\\
@VVV@Vj^*VV\\
H^n(X,W)@>i^*>>H^n(Y,W')
\end{CD}\quad\text{and}\quad
\begin{CD}
H_n(Y,W')@>i_*>>H_n(X,W)\\
@Vj_*VV@VVV\\
H_n(Y,\,Z\cup W')@>>>H_n(X,\,Z\cup W),
\end{CD}\]
\smallskip

(a) the upper left cohomology group surjects onto the pullback of $i^*$ and $j^*$;

(b) the pushout of $i_*$ and $j_*$ injects into the lower right homology group.
\end{lemma}

\begin{proof} We will prove part (a); part (b), which is not used, can be proved similarly.

We need to show that if $i^*(\alpha)=j^*(\beta)$, then $\alpha$ and $\beta$ have a common point-inverse.

Let us triangulate $X$ so that $Y$, $Z$ and $W$ are subcomplexes, and let us represent $\alpha$ and
$\beta$ by simplicial cocycles $x$ and $y$.
Then $i^*(x)-j^*(y)=\delta c$ for some $c\in C^{n-1}(Y,W')$.
The cochain-level $i^*$ is clearly surjective, so there exists a $\tilde c\in C^{n-1}(X,W)$ such that 
$i^*(\tilde c)=c$.
Let $x'=x-\delta\tilde c$.
Then $x'$ is a cocycle and $i^*(x')=j^*(y)$.
Since the cochain-level version of the diagram is obviously a pullback diagram, $x'$ and $y$ have 
a common point-inverse $z\in C^n(X,\,Z\cup W)$.
It must be a cocycle, since it maps to the cocycle $x'$ under a monomorphism of chain complexes.
Then clearly $[z]$ is the desired common point-inverse of $[x']=[x]=\alpha$ and $[y]=\beta$.
\end{proof}

\subsection{Kernel of $\tau$: homology}

\begin{theorem}\label{10.13'} Let $X$ be a coronated polyhedron and let $\dots\xr{f_2} R_2\xr{f_1} R_1$ be 
its sequential resolution.
Then

(a) $p^1H_n(X)\simeq\derlimfg H_{n+1}(R_i)$;

(b) there is a short exact sequence
\[0\xr{\phantom{\tau}}\derlim H_{n+1}(R_i)/\derlimfg H_{n+1}(R_i)\xr{\phantom{\tau}}
pH_n(X)\xr{\tau_n^X}qH_n(X)\xr{\phantom{\tau}}0.\]
\end{theorem}

\begin{proof} 
Like in \cite{M-V}*{Proof of Theorem \ref{cor:main-ses2}}, it suffices to consider only one specific sequential 
resolution of $X$, which can be taken to be a pre-compactohedral inverse sequence.
Let $G_i=H_{n+1}(R_i)$.
Like in the proof of Theorem \ref{10.13}, by Theorem \ref{10.11} it suffices to prove that the image of 
$\tau_{n+1}^{(1)}\:p^1H_{n+1}(X)\to q^1H_{n+1}(X)$ coincides with the subgroup $\derlimfg G_i$ of 
$\derlim G_i=q^1H_{n+1}(X)$.

Let $Q\subset X$ be a compactum, and let us represent $Q$ as the limit of an inverse sequence 
$\dots\to Q_1\to Q_0$ of compact polyhedra.
Since the latter is a resolution of $Q$ (see \cite{M00}*{Lemma \ref{book:compact-resolution}}),
the inclusion $Q\subset X$ gives rise to an increasing sequence $(n_i)$ and maps $g_i\:Q_{n_i}\to R_i$ 
which commute up to homotopy with the bonding maps and with the projections of the inverse limits 
(see \cite{M00}*{Corollary \ref{book:shape}}).
These in turn yield a map $\derlim H_{n+1}(Q_{n_i})\to\derlim H_{n+1}(R_i)$, which is identified
with the inclusion induced map $q^1H_{n+1}(Q)\to q^1H_{n+1}(X)$.
Since the $Q_i$ are compact, $\derlim H_{n+1}(Q_{n_i})=\derlimfg H_{n+1}(Q_{n_i})$, and it follows that 
the image of $\tau_{n+1}^{(1)}\:p^1H_{n+1}(X)=\colim_Q q^1H_{n+1}(Q)\to q^1H_{n+1}(X)$ lies in the subgroup 
$\derlimfg G_i$.

Conversely, to show that $\tau_{n+1}^{(1)}$ is onto $\derlimfg G_i$ it suffices (by arguments similar
to those in the proof of Theorem \ref{10.13}) to prove the following lemma.

\begin{lemma} \label{10.14'} 
Let $X$ be a coronated polyhedron and let $\dots\xr{p_2} R_2\xr{p_1} R_1$ be a pre-compactohedral 
resolution of $X$.
Suppose that we are given finitely generated subgroups $F_i\subset H_n(R_i)$ such that each 
$(p_i)_*(F_{i+1})\subset F_i$.
Then there exist compact subpolyhedra $Q_i\subset R_i$ such that each $p_i(Q_{i+1})\subset Q_i$ and
each $F_i$ lies in the image of $H_n(Q_i)$. 
\end{lemma}

\begin{proof}
Let $K_i\subset L_i\subset R_i$ be the associated compact and closed polyhedra.
Since each generator of $F_1$ can be represented by a finite simplicial cycle, there exists a compact
subpolyhedron $A_1\subset R_1$ such that $F_1$ lies in the image of $H_n(A_1)$.

Let us assume that $A_k$ is a compact subpolyhedron of $R_k$ such that $F_k$ lies in the image of $H_n(A_k)$. 
Then the image of $F_k$ in $H_n(R_k,A_k)$ is trivial.
Therefore the image of $F_k$ in $H_n(R_k,\,L_k\cup A_k)$ is also trivial.
Let $E_{k+1}=p_k^{-1}(L_k\cup A_k)$.
Since $p_k$ restricts to a homeomorphism between $R_{k+1}\but E_{k+1}$ and $R_k\but(L_k\cup A_k)$, and even 
between their closures, the map of pairs $(R_{k+1},\,E_{k+1})\to\big(R_k,\,L_k\cup A_k)$ given by $p_k$ induces 
an isomorphism in homology.
So in the commutative diagram
\[\begin{CD}
F_{k+1}@>>>H_n(R_{k+1},\,E_{k+1})\\
@VVV@VVV\\
F_k@>>>H_n(R_k,\,L_k\cup A_k)
\end{CD}\]
the right vertical arrow is an isomorphism and the bottom horizontal arrow is trivial.
Hence the top horizontal arrow is also trivial.
Therefore $F_{k+1}$ lies in the image of $H_n(E_{k+1})$.
Then each generator $\alpha_i\in F_{k+1}$ is the image of some $\beta_i\in H_n(E_{k+1})$, which
can be represented by a finite simplicial cycle.
Hence there exists a compact subpolyhedron $A_{k+1}\subset E_{k+1}$ such that 
$F_{k+1}$ lies in the image of $H_n(A_{k+1})$.

Since $A_{k+1}\subset E_{k+1}$, we have $p_k(A_{k+1})\subset L_k\cup A_k$.
Consequently $p_k\big(p_{k+1}(A_{k+2})\big)\subset p_k(L_{k+1}\cup A_{k+1})\subset K_k\cup p_k(A_{k+1})$.
Let $B_k=A_k\cup p_k(A_{k+1})$.
Then \[p_k(B_{k+1})=p_k(A_{k+1})\cup p_k\big(p_{k+1}(A_{k+2})\big)\subset K_k\cup p_k(A_{k+1})\subset K_k\cup B_k.\]
Finally, let $Q_k=K_k\cup B_k$.
Then $p_k(Q_{k+1})=p_k(K_{k+1})\cup p_k(B_{k+1})\subset K_k\cup B_k=Q_k$. 
On the other hand, $Q_k=K_k\cup A_k\cup p_k(A_{k+1})$ is a compact subpolyhedron of $R_k$, and $F_k$ lies 
in the image of $H_n(Q_k)$ since $A_k\subset Q_k$.
\end{proof}
\end{proof}

\section{Topology of $\lim^2$}

\subsection{Vanishing of $\lim^2$} \label{lim2van}

By a well-known result of R. Goblot, $\lim^{p+2} G_\alpha=0$ for any inverse system of abelian groups indexed 
by a directed set of cofinality $\le\aleph_p$ (see \cite{Mard} or \cite{Je}).
In particular, the Continuum Hypothesis (in the presence of the Axiom of Choice) implies the vanishing of 
$\invlim^p G_\alpha$ for $p\ge 3$ for any inverse system of abelian groups $G_\alpha$ indexed by a directed set 
of cofinality at most continuum.

\begin{lemma} Let $X$ be a separable metrizable space.

(a) The set of all compact subsets of $X$ has cardinality at most continuum.

(b) $X$ has an ANR expansion of cardinality at most continuum.
\end{lemma}

\begin{proof}
Let $Q_i$ be the standard cubulation of the cube $[0,2^i]^i$ into $(2^i)^i$ cubes with integer vertices.
Every compact subset of the Hilbert cube $[0,1]^\infty$ can be identified with the inverse limit of the minimal
subcomplexes $K_i$ of the cubulations $2^{-i}Q_i$ of $[0,1]^i$ containing the image of $X$.
The finite cubical complex $2^{-i}Q_i$ has only finitely many distinct subcomplexes, and it follows that 
the set of compact subsets of the Hilbert cube $I^\infty$ has cardinality at most continuum.
(In fact, precisely continuum, since there are always at least two options for each $K_i$, after 
$K_1,\dots,K_{i-1}$ have been selected.)
Since there exists an embedding $g$ of $X$ in $I^\infty$, we immediately get (a).
Also the set of open neighborhoods of $g(X)$ has cardinality at most continuum, since their complements are
compact subsets of $I^\infty$.
This implies (b).
\end{proof}

\begin{corollary}
If $X$ is a separable metrizable space and $i\ge 0$, the following statements cannot be disproved in ZFC:

(a) $\lim^p H_i(P_\beta)=0$ for $p\ge 3$, where the $P_\beta$ form an ANR expansion of $X$;

(b) $\lim^p H^i(K_\alpha)=0$ for $p\ge 3$, where $K_\alpha$ run over all compact subsets of $X$.
\end{corollary}

This suggests the slogan ``$\lim^3,\lim^4,\dots$ carry no geometrically relevant information''.

\begin{proposition} \label{lim2vanishing}
If $G_\alpha$ is an inverse system of abelian groups indexed by a directed set $\Lambda$ of cofinality 
$\le\aleph_p$ and such that $\lim G_\alpha$ surjects onto each $G_\alpha$, then $\lim^{p+1} G_\alpha=0$.
\end{proposition}

In particular, for an inverse system of abelian groups $G_\alpha$ indexed by a directed set of cofinality 
at most continuum and such that $\lim G_\alpha$ surjects onto each $G_\alpha$, the vanishing of 
$\lim^2 G_\alpha$ cannot be disproved in ZFC.

In topology, the requirement that $\lim G_\alpha$ surject onto each $G_\alpha$ is natural for chains and
cochains, rather than homology and cohomology.
This will be discussed in more detail in the next subsection (\S\ref{conjectures}).

\begin{proof} Let $L=\lim G_\alpha$, and let $K_\alpha$ be the kernel of $p^\infty_\alpha\:L\to G_\alpha$. 
Then the short exact sequence of inverse systems
\[0\to K_\alpha\to L\to G_\alpha\to 0\]
yields $\lim^{p+1} G_\alpha\simeq\lim^{p+2} K_\alpha$ for each $p\ge 0$, due to $\lim^i L=0$ for all $i>0$.
Thus the assertion follows from Goblot's theorem.
\end{proof}

\begin{example} \label{sklyarenko}
As observed by Sklyarienko \cite{Sk84}*{Example 1.2}, an inverse system of abelian groups 
$G_\alpha$ indexed by a directed poset of cofinality continuum and such that $\lim G_\alpha$ surjects onto 
each $G_\alpha$ may have $\lim^1 G_\alpha\ne 0$ (in ZFC).

Clearly, $\Pi\bydef \prod_{n\in\N}\Z$ is the direct limit of its finitely generated subgroups.
It is well-known that every finitely generated subgroup of $\Pi$ is contained in a finitely generated 
direct summand of $\Pi$ (see \cite{Fu}*{Proof of Proposition 87.4 and Theorem 85.1}).
Hence $\Pi$ is also the direct limit of its finitely generated direct summands $F_\alpha$,
where any two, $F_\alpha$ and $F_\beta$, are contained in a third one, $F_\gamma$.
Then $\Hom(\Pi,\Z)$ is the inverse limit of the groups $G_\alpha\bydef \Hom(F_\alpha,\Z)$. 
Since each inclusion $F_\alpha\to\Pi$ is a split injection, each map $\Hom(\Pi,\Z)\to\Hom(F_\alpha,\Z)$
is a split surjection (in particular, a surjection).
Since the $F_\alpha$ are finitely generated free abelian groups, 
$\lim^1 G_\alpha\simeq\Ext(\colim F_\alpha,\,\Z)\simeq\Ext(\Pi,\Z)$.
But it is well-known that $\Ext(\Pi,\Z)\ne 0$ (see references in \cite{M00}*{Remark \ref{book:UCF}(2)}).
\end{example}

\begin{example} \label{kuzminov}
Kuz'minov constructed (long before \cite{Sk84}) an inverse system of abelian groups $K_\alpha$ indexed 
by a directed poset of cofinality continuum such that $\invlim^2 K_\alpha\ne 0$ (in ZFC) \cite{Ku}*{\S10} 
(see also \cite{Sk84}*{Remark after Theorem 1.4}).

Namely, in the notation of the previous example (Example \ref{sklyarenko}),
$K_\alpha=\Hom(\Pi/F_\alpha,\,\Z)$.
The short exact sequence of direct systems $0\to F_\alpha\to\Pi\to\Pi/F_\alpha\to 0$ yields
a short exact sequence of inverse systems $0\to\Hom(\Pi/F_\alpha,\Z)\to\Hom(\Pi,\Z)\to\Hom(F_\alpha,\Z)\to 0$
(the exactness in the last term is due to each $F_\alpha$ being a direct summand in $\Pi$),
which can also be written as $0\to K_\alpha\to\Sigma\to G_\alpha\to 0$, where
$\Sigma=\bigoplus_{n\in\N}\Z$.
Since $\lim^i\Sigma=0$ for all $i>0$, we have $\lim^2 K_\alpha\simeq\lim^1 G_\alpha\ne 0$.
\end{example}

\begin{remark} \label{kuzminov+}
The groups $K_\alpha$ of the previous example (Example \ref{kuzminov}) can be equivalently described as 
arbitrary direct summands of the group $\Sigma=\bigoplus_{n\in\N}\Z$ with finitely generated complements.

Indeed, if $H$ is a subgroup of a group $G$, let $H^*$ be the subgroup of $\Hom(G,\Z)$ consisting of all 
homomorphisms whose kernel contains $H$.
The quotient map $G\to G/H$ induces an injection $\Hom(G/H,\,\Z)\to\Hom(G,\Z)$ whose image is clearly $H^*$.
If $H$ is a direct summand of $G$, then 
$\Hom(G,\Z)\simeq\Hom\big(H\oplus (G/H),\,\Z\big)\simeq\Hom(H,\Z)\oplus\Hom(G/H,\,\Z)\simeq\Hom(H,\Z)\oplus H^*$, 
so $H^*$ is a direct summand of $\Hom(G,\Z)$.
If $H$ is a finitely generated direct summand, then $H^*$ has a finitely generated complement, namely, 
$\Hom(H,\Z)$.
Conversely, if $H$ is a direct summand with finitely generated complement, that is, $G/H$ is finitely generated, 
then $H^*\simeq\Hom(G/H,\,\Z)$ is finitely generated.

Let $\Phi\:G\to\Hom\big(\Hom(G,\,\Z),\,\Z\big)$ be the natural inclusion given by $\Phi(g)(f)=f(g)$.
If $g\in H$ and $f\in H^*$, then $\Phi(g)(f)=0$.
Thus $\Phi(H)\subset H^{**}$.
Now assume that $\Phi$ is an isomorphism (which is indeed the case for $G=\Sigma,\Pi$ \cite{Fu}*{Corollary 94.6}),
that $H$ is a direct summand in $G$, and that every $g\notin H$ lies in a direct summand $F$ containing $H$
as a direct summand of $F$ and such that $F/H$ is free abelian.
(The latter hypothesis trivially holds in the case $G=\Sigma$, by setting $F=G$; whereas in the case where 
$G=\Pi$ and $H$ is finitely generated we use that the subgroup generated by $H$ and $g$ is contained in 
a finitely generated direct summand $F$, and that $G=H\oplus C$ and $H\subset F$ imply $F=H\oplus (F\cap C)$.)
Then for each $g\notin H$ there exists an $f\in H^*$ such that $f(g)\ne 0$.
Thus $\Phi(g)\notin H^{**}$.
Since $\Phi$ is an isomorphism, we conclude that $\Phi(H)=H^{**}$. 
\end{remark}

\begin{remark} \label{kuzminov++}
It is not hard to see that Kuz'minov's inverse system $\{K_\alpha\}$ (see Example \ref{kuzminov} and 
Remark \ref{kuzminov+}) is cofinal in the inverse system $\{K_\beta\}$ of all subgroups $K_\beta\subset\Sigma$ 
such that $\Sigma/K_\beta$ is finitely generated, and consequently $\lim^2 K_\beta=\lim^2 K_\alpha\ne 0$.

Indeed, let $G\subset\Sigma$ be such that $\Sigma/G$ is finitely generated.
Since the quotient map $q\:\Sigma\to\Sigma/G$ is surjective, there exists a finite $S\subset\Sigma$ 
such that $q(S)$ generates $\Sigma/G$.
Let $B$ be a set of free generators of $\Sigma$.
Since each element of $S$ is a finite $\Z$-linear combination of elements of $B$, we have 
$S\subset\left<A\right>$ for some finite $A\subset B$.
Then $q(A)$ also generates $\Sigma/G$.
Hence for each $g\in\Sigma$ we have $q(g)=q(a_g)$ for some $a_g\in\left<A\right>$.
Let $H=\left<x-a_x\mid x\in B\but A\right>$.
Then $q(H)=\{0\}$ and so $H\subset G$.
On the other hand $\Sigma=H\oplus\left<A\right>$, where $A$ is finite.

Let us also note that since $\{K_\alpha\}$ is cofinal in $\{K_\beta\}$ and its indexing poset is directed,
the indexing poset of $\{K_\beta\}$ is also directed.
\end{remark}

\begin{example} \label{lim2example}
There exists a separable metrizable space $X$ that is a union of a directed collection of 
its compact subsets $X_\beta$ such that $\lim^2 H^1(X_\beta)\ne 0$.
(It remains unclear whether $\lim^2 H^1(K)\ne 0$ over all compact $K\subset X$.)

Let $\Sigma$ be a countable infinite set.
Let $Q$ be the product $\prod_{s\in\Sigma} D^2$ of copies of the $2$-disc.
The linear map $D^2=c*S^1\to 0*1=[0,1]=I$ (extending the constant maps $c\to 0$ and $S^1\to 1$) yields a map 
$\pi\:Q\to I^\infty$ onto $I^\infty=\prod_{s\in\Sigma} I$.
Of course, both $Q$ and $I^\infty$ are copies of the Hilbert cube. 

Every subset $A\subset\Sigma$ can be identified with a vertex $|A|$ of $I^\infty$, namely the one given
by the indicator function of $A$, that is, $|A|=(t_s)_{s\in\Sigma}$, where $t_s=1$ if $s\in A$ and 
$t_s=0$ if $s\notin A$.
If $C=(A_1\subset\dots\subset A_n)$ is a nonempty finite chain of subsets of $\Sigma$, let $|C|$ denote
the convex hull of $\{|A_1|,\dots,|A_n|\}$ in $I^\infty$.
If $S$ is a collection of subsets of $\Sigma$, let $|S|$ denote the union of $|C|$'s over all nonempty
finite chains $C=(A_1\subset\dots\subset A_n)$ with each $A_i$ taken from $S$.
If $S$ is a finite set of cardinality $m$, then $|S|$ is a triangulated $m$-dimensional cube.
(Compare the geometric realization of a poset in \cite{M3}.)  

For each $A\subset\Sigma$ let $T^A$ denote the subset $\prod_{i\in A} S^1\times\prod_{i\notin A} c$ of $Q$.
Clearly, $T^A=\pi^{-1}(|A|)$.
Let us note that any point $(z_a)_{a\in A}\in T^A$ can be identified with a map $A\to S^1$, namely 
the one given by $a\mapsto z_a$. 
If $B$ is a subset of $A$, we have the natural projection $\pi^A_B\:T^A\to T^B$, which sends 
any map $f\:A\to S^1$ into its restriction $f|_B\:B\to S^1$.

Given a nonempty finite chain $C=(A_1\subset\dots\subset A_n)$ of subsets of $\Sigma$, let $T^C=\pi^{-1}(|C|)$.
For the two-element chain $C=(B\subset A)$, it is easy to see that $T^C$ is homeomorphic to
the mapping cylinder $\cyl\big(\pi^A_B\big)$.
In general, it is not hard to see that $T^C$ is homeomorphic to the iterated mapping cylinder 
$\cyl(\pi^{A_n}_{A_{n-1}},\dots,\pi^{A_2}_{A_1})$.
(The iterated mapping cylinder $\cyl(f_1,\dots,f_n)$ of a sequence of composable maps 
$X_0\xr{f_1}\dots\xr{f_n}X_n$ is defined inductively as the mapping cylinder of the composition 
$\cyl(f_1,\dots,f_{n-1})\xr{p} X_{n-1}\xr{f_n} X_n$, where $p$ is the natural projection.)
If $S$ is a collection of subsets of $\Sigma$, let $T^S=\pi^{-1}(|S|)$.
Clearly, $T^S$ is the union of $T^C$'s over all nonempty finite chains $C=(A_1\subset\dots\subset A_n)$ 
with each $A_i$ taken from $S$.

From now on, we assume that $\Sigma$ is in fact the countably generated free abelian group $\Sigma=\bigoplus_\N\Z$.
If $G$ is a subgroup of $\Sigma$, then the topological group $\Hom(G,S^1)$ can be regarded as a subset of the
set of maps $G\to S^1$ and so gets identified with a subset $R^G$ of $T^G$.
If $G$ is freely generated by a set $A\subset G$, then the projection $\pi^G_A\:T^G\to T^A$ sends $R^G$ 
homeomorphically onto $T^A$.
If $F$ is a subgroup of $G$, then the projection $\pi^G_F\:T^G\to T^F$ restricts to a map
$\rho^G_F\:R^G\to R^F$ (since the restriction of a homomorphism to a subgroup is a homomorphism).
For the two-element chain $C=(F\subset G)$, let $R^C\subset T^C$ be the mapping cylinder $\cyl(\rho^G_F)$.
In general, given a nonempty finite chain $C=(G_1\subset\dots\subset G_n)$ of subgroups of $\Sigma$, 
let $R^C\subset T^C$ be the iterated mapping cylinder $\cyl(\rho^{G_n}_{G_{n-1}},\dots,\rho^{G_2}_{G_1})$.
If $S$ is a collection of subgroups of $\Sigma$, let $X^S\subset T^S$ be the union of $X^C$'s over all 
nonempty finite chains $C=(G_1\subset\dots\subset G_n)$ with each $G_i$ taken from $S$.

Finally, let $W$ be the collection of all direct summands $K_\beta$ of $\Sigma$ such that $\Sigma/K_\beta$ 
is finitely generated.
Given a $K_\beta\in W$, let $V_\beta\subset W$ consist of all $K_\beta$'s that contain $K_\beta$.
Then each $X_\beta\bydef R^{V_\beta}$ is compact, and $X\bydef R^W$ is the union of all $R^{V_\beta}$'s.

We claim that each $R^{V_\beta}$ deformation retracts onto $R^{K_\beta}$.
Indeed, given a closed subset $A\subset X$ whose inclusion in the space $X$ is a cofibration (or equivalently 
$A$ is a neighborhood deformation retract of $X$) and a continuous map $f\:X\to Y$, then its mapping cylinder
$\cyl(f)$ deformation retracts onto the mapping cylinder $\cyl(f|_A)$ of the restriction $f|_A\:A\to Y$.
It follows that the iterated mapping cylinder $\cyl(f_1,\dots,f_n)$ of a sequence of composable maps 
$X_0\xr{f_1}\dots\xr{f_n}X_n$ deformation retracts onto the union of all its iterated mapping subcylinders 
$\cyl(f_{i_1},\dots,f_{i_k})$ excluding $\cyl(f_1,\dots,f_n)$ itself and one ``free face''
$\cyl(f_1,\dots,f_{n-1})$.

An inductive application of this elementary deformation retraction yields a deformation retraction of
each $R^{V_\beta}$ onto $R^{K_\beta}$.
Namely, the triangulated cube $|V_\beta|$ is a simplicial cone with cone vertex $|K_\beta|$.
Hence there is a simplicial collapse of $|V_\beta|$ onto $|K_\beta|$, with all intermediate complexes being
subcones of $|V_\beta|$.
By echoing each elementary simplicial collapse in this sequence with an elementary deformation retraction as above, 
we get the desired deformation retraction of $R^{V_\beta}$ onto $R^{K_\beta}$.

If $G$ is a subgroup of $\Sigma$, clearly $H^1(R^G)\simeq G$.
Moreover, if $F$ is a subgroup of $G$, then $\rho^G_F\:R^G\to R^F$ induces the inclusion $F\to G$ on $H^1$.
Thus each $H^1(R^{V_\beta})\simeq H^1(R^{K_\beta})\simeq K_\beta$, and if $K_\alpha$ is a subgroup of
$K_\beta$, the inclusion $R^{V_\beta}\subset R^{V_\alpha}$ induces the inclusion $K_\alpha\to K_\beta$ on $H^1$.
Hence $\lim^2 H^1(R^{V_\beta})\simeq\lim^2 K_\beta\ne 0$ by the previous example and remarks
(Example \ref{kuzminov} and Remarks \ref{kuzminov+}, \ref{kuzminov++}; the remarks are needed only to simplify 
the description of $R^W$).
\end{example}

\subsection{$\lim^2$ and $\tau$}\label{conjectures}

\begin{theorem} \label{lim2theorem}
Assume that 

(1) $\lim^1 G_\alpha=0$, where $G_\alpha$ is the Marde\v si\'c--Prasolov inverse system (see Example \ref{MP}) and

(2) $\lim^3 H_\beta=0$ for every inverse system $H_\beta$ of cofinality at most continuum, where each
$H_\alpha$ is an abelian group of cardinality at most continuum.

Let $X$ be a local compactum; let $K_i=\ker\big(H_i(X)\to qH_i(X)\big)$.
Then there is an exact sequence
\[0\to p^1 H_{i+1}(X)\to K_i\to pH_i(X)\xr{\tau_n^X} qH_i(X)\to q^2H_{i+1}(X)\to K_{i-1}\to q^1H_i(X)\to 0.\]
\end{theorem}

The proof is based on the chain complexes for Steenrod--Sitnikov homology from \cite{M-II}.
A part of the proof is similar to that of \cite{Md}*{Satz A}. 

\begin{proof}
Since $X$ is a local compactum, it is the limit of an inverse sequence $\dots\xr{p_1} P_1\xr{p_0} P_0$
of locally compact separable polyhedra and proper maps (see \cite{M00}*{Theorem \ref{book:isbell}}, see also
\cite{M00}*{Remark \ref{book:nerves and cubes}}).
Moreover, the $P_i$ come endowed with triangulations or cubulations $L_i$ such that each $p_i$ is simplicial 
or cubical as a map from $L_{i+1}$ to a certain subdivision of $L_i$.
Using the product cell structure of $L_i\x [i-1,\,i]$, we make the infinite mapping telescope
$P_{[0,\infty)}$ into a cell complex $L$ (that is, a cone complex whose cones are PL balls), see 
\cite{M00}*{Theorem \ref{book:telescope-theorem}}.
Thus $P_{[0,\infty]}=X\cup L$.

Let $\kappa'$ be the set of all mapping subtelescopes $K_{[0,\infty)}\subset P_{[0,\infty)}$ such that 
each $K_i$ is a finite subcomplex of $L_i$, and let $\bar\nu'$ be the set of all mapping subtelescopes 
$N_{[0,\infty)}\subset P_{[0,\infty)}$ such that each $N_i$ is a subcomplex of $L_i$ and $\lim N_i=\emptyset$.
Let $\nu'=\{\Cl(L\but F)\mid F\in\bar\nu'\}$.

For a subcomplex $A$ of $L$ and a subcomplex $B$ of $A$ let 
\[\resizebox{\linewidth}{!}{
\begin{minipage}{\linewidth}
\begin{alignat*}{4}
C_*^{\kappa'}(A,B)&\ =\ \colim_{K\in\kappa'}&C_*^\infty(A\cap K,\,B\cap K)&\ \simeq\ 
\colim_{K\in\kappa'}\ \lim_{N\in\nu'}\,C_*\big(A\cap K,\,(B\cap K)\cup (A\cap K\cap N)\big),\\
C_*^{\nu'}(A,B)&\ =\ \lim_{N\in\nu'}&C_*\big(A,\,B\cup (A\cap N)\big)&\ \simeq\ 
\lim_{N\in\nu'}\ \colim_{K\in\kappa'}\,C_*\big(A\cap K,\,(B\cap K)\cup (A\cap K\cap N)\big).\\
\end{alignat*}
\end{minipage}}\]
Then $C_*^{\kappa'}(A,B)\simeq C_*^{\nu'}(A,B)$ and 
$H_i(X)\simeq H_{i+1}\big(C_*^{\nu'}(L,L_0)\big)$ \cite{M-II}.
Let us write $C_*^\infty=C_*^{\nu'}(L,L_0)$ and $C_*^N=C_*(L,\,N\cup L_0)$.
Then $C_*^\infty=\lim_{N\in\nu'}C_*^N$ and $H_i(X)\simeq H_{i+1}(C_*^\infty)$.

Given a $K\in\kappa'$, what was just said of $L=L_{[0,\infty)}$ applies to $K=K_{[0,\infty)}$ as well.
Thus writing $C_*^{\infty K}=C_*^{\nu'}(K,K_0)$ and $C_*^{NK}=C_*\big(K,\,(N\cap K)\cup K_0\big)$, we have 
$C_*^{\infty K}=\lim_{N\in\nu'} C_*^{NK}$ and $H_i(\lim K_i)\simeq H_{i+1}(C_*^{\infty K})$.

Let $Z_i^N=\ker\big(C_i^N\xr{\partial} C_{i-1}^N\big)$, $B_i^N=\im\big(C_{i+1}^N\xr{\partial} C_i^N\big)$ and
$H_i^N=H_i(C_*^N)=Z_i^N/B_i^N$.
Since $L$ deformation retracts onto $L_0$, we have $H_i(L,L_0)=0$ for each $i$.
Hence from the long exact sequence of the triple $(L,\,N\cup L_0,\,L_0)$
we get $H_i(N)\simeq H_{i+1}(L,\,N\cup L_0)=H_{i+1}^N$.
Therefore $qH_i(X)\simeq\lim_{N\in\nu'} H_i(N)\simeq\lim_{N\in\nu'}H_{i+1}^N$.
Similarly $q^kH_i(X)\simeq\lim^k_{N\in\nu'}H_{i+1}^N$.
 
The short exact sequences \[0\to Z_i^N\to C_i^N\to B_{i-1}^N\to 0,\]
\[0\to B_i^N\to Z_i^N\to H_i^N\to 0\]
give rise to the long exact sequences of derived functors

\[\dots\to\lim\nolimits^k C_i^N\to\lim\nolimits^k B_{i-1}^N\to
\lim\nolimits^{k+1} Z_i^N\to \lim\nolimits^{k+1} C_i^N\to\dots\tag{I}\]
\[\dots\to\lim\nolimits^k B_i^N\to\lim\nolimits^k Z_i^N\to
\lim\nolimits^k H_i^N\to\lim\nolimits^{k+1} B_i^N\to\dots\tag{II}\]
\smallskip

It is easy to see that the inverse system $C_i^N$ is isomorphic to the Marde\v si\'c--Prasolov inverse system
$G_\alpha$.
Hence by our hypothesis (1), $\lim^1 C_i^N=0$.
Moreover, each homomorphism $\lim C_i^N\to C_i^{N_0}$ is split surjective, and consequently
$\lim B_i^N\to B_i^{N_0}$ is also surjective.
Hence by the proof of Proposition \ref{lim2vanishing}, $\lim^2 C_i^N$ and $\lim^2 B_i^N$
are isomorphic to $\lim^3$ of some inverse systems of the same cofinality and of cardinality at most continuum.
Consequently, our hypothesis (2) implies $\lim^2 C_i^N=0$ and $\lim^2 B_i^N=\lim^3 B_i^N=0$.

Then a leftmost fragment of (I) reduces to 
\[0\to\lim Z_i^N\to\lim C_i^N\to\lim B_{i-1}^N\to\lim\nolimits^1 Z_i^N\to 0\tag{I$'$}\]
along with an isomorphism 
\[\lim\nolimits^1 B_{i-1}^N\simeq\lim\nolimits^2 Z_i^N.\tag{I$''$}\]
Similarly, a leftmost fragment of (III) similarly reduces to
\[0\to\lim B_i^N\to\lim Z_i^N\to\lim H_i^N\to
\lim\nolimits^1 B_i^N\to\lim\nolimits^1 Z_i^N\to\lim\nolimits^1 H_i^N\to 0\tag{II$'$}\]
along with an isomorphism
\[\lim\nolimits^2 Z_i^N\simeq\lim\nolimits^2 H_i^N.\tag{II$''$}\]

Let $Z_i^\infty=\ker\big(C_i^\infty\xr{\lim\partial} C_{i-1}^\infty\big)$,
$B_i^\infty=\im\big(C_{i+1}^\infty\xr{\lim\partial} C_i^\infty\big)$
and $H_i^\infty=Z_i^\infty/B_i^\infty$.
By the above $H_i(X)\simeq H_i(C_*^\infty)=H_{i+1}^\infty$.

The boundary map $\partial\:C_i^N\to C_{i-1}^N$ is the composition 
$C_i^N\twoheadrightarrow B_{i-1}^N\rightarrowtail C_{i-1}^N$ of a surjection and an injection.
Hence the boundary map $\lim\partial\:\lim C_i^N\to\lim C_{i-1}^N$ is of the form
$\lim C_i^N\xr{f}\lim B_{i-1}^N\rightarrowtail\lim C_{i-1}^N$.
Therefore $Z_i^\infty=\ker\lim\partial=\ker f$ and $B_{i-1}^\infty=\im\lim\partial\simeq\im f$.
Taking into account (I$'$), we get \[Z_i^\infty\simeq\lim Z_i^N\tag{III}\] and
a short exact sequence \[0\to B_{i-1}^\infty\to\lim B_{i-1}^N\to\lim\nolimits^1 Z_i^N\to 0.\tag{III$'$}\]

On the other hand, the third isomorphism theorem $(G/K)/(H/K)\simeq G/H$ for abelian groups $K\subset H\subset G$
yields a short exact sequence
\[0\to\lim B_i^N/B_i^\infty\to\lim Z_i^N/B_i^\infty\to\lim Z_i^N/\lim B_i^N\to 0,\tag{IV}\]
where the inclusions $B_i^\infty\subset\lim B_i^N\subset\lim Z_i^N$ follow from (II$'$) and (III$'$).
Using (III) and (III$'$) we can rewrite (IV) as
\[0\to\lim\nolimits^1 Z_{i+1}^N\to H_i^\infty\to\lim Z_i^N/\lim B_i^N\to 0.\tag{IV$'$}\]
Combining (II$'$) with (IV$'$) and using the isomorphisms (I$''$) and (II$''$), we get
\[0\to\lim\nolimits^1 Z_{i+1}^N\to H_i^\infty\to\lim H_i^N\to
\lim\nolimits^2 H_{i+1}^N\to\lim\nolimits^1 Z_i^N\to\lim\nolimits^1 H_i^N\to 0.\tag{V}\]

We can also write the analogue of (V) for each compact $K\subset X$.
To this end let 
$Z_i^{NK}=\ker\big(C_i^{NK}\xr{\partial} C_{i-1}^{NK}\big)$, 
$B_i^{NK}=\im\big(C_{i+1}^{NK}\xr{\partial} C_i^{NK}\big)$ and
$H_i^{NK}=H_i(C_*^{NK})=Z_i^{NK}/B_i^{NK}$.
These now have a countable indexing poset, so, in particular, $\lim^2 H_i^{NK}=0$.
Also, by the above $q^kH_i(\lim K_i)\simeq\lim^k_{N\in\nu'}H_{i+1}^{NK}$.
Hence $p^k H_i(X)\simeq\colim_{K\in\kappa'}q^kH_i(\lim K_i)\simeq
\colim_{K\in\kappa'}\lim^k_{N\in\nu'}H_{i+1}^{NK}$.
We also have $H_i(X)\simeq\colim_{K\in\kappa'} H_i(\lim K_i)$ by the definition of Steenrod--Sitnikov homology.
Since $H_i(X)\simeq H_i(C_*^\infty)=H_{i+1}^\infty$ and 
$H_i(\lim K_i)\simeq H_{i+1}^{\infty K}\bydef H_{i+1}(C_*^{\infty K})$,
we get $H_i^\infty\simeq\colim H_i^{\infty K}$.

Upon passing to the direct limit over all compact $K\subset X$, we obtain a commutative diagram
\[\scalebox{0.98}{$\begin{CD}
0@>>>\colim\lim\nolimits^1 Z_{i+1}^{N K}@>>>\colim H_i^{\infty K}@>>>\colim\lim H_i^{N K}@>>>0\\
@.@VVV@VVV@VVV@VVV\\
0@>>>\lim\nolimits^1 Z_{i+1}^N@>>>H_i^\infty@>>>\lim H_i^N@>>>\lim\nolimits^2 H_{i+1}^N
\end{CD}$}
\tag{VI}\]
continued as
\[
\begin{CD}
0@>>>\colim\lim\nolimits^1 Z_i^{N K}@>>>\colim\lim\nolimits^1 H_i^{N K}@>>>0\\
@VVV@VVV@VVV@.\\
\lim\nolimits^2 H_{i+1}^N@>>>\lim\nolimits^1 Z_i^N@>>>\lim\nolimits^1 H_i^N@>>> 0
\end{CD} 
\tag{VI$'$}\]
By the above, the vertical map $\colim H_i^{\infty K}\to H_i^\infty$ in the diagram (VI) is an isomorphism.
Hence $\phi\:\colim\lim\nolimits^1 Z_{i+1}^{N K}\to\lim\nolimits^1 Z_{i+1}^N$ is an injection,
$\psi\:\colim\lim H_i^{N K}\to\lim H_i^N$ is a surjection onto the image of the horizontal map
$H_i^\infty\to\lim H_i^N$, and $\coker\phi\simeq\ker\psi$ (using the third 
isomorphism theorem $(G/K)/(H/K)\simeq G/H$).
From (VI$'$) we also have $\colim\lim\nolimits^1 Z_i^{N K}\simeq\colim\lim\nolimits^1 H_i^{N K}$.
Using these observations, we can replace the leftmost half of (V) with something different:
\begin{multline}
0\to\colim\lim\nolimits^1 H_{i+1}^{N K}\to\lim\nolimits^1 Z_{i+1}^N\to\colim\lim H_i^{N K}\to\lim H_i^N\\
\to\lim\nolimits^2 H_{i+1}^N\to\lim\nolimits^1 Z_i^N\to\lim\nolimits^1 H_i^N\to 0.\tag{VII}
\end{multline}
Using the notation of \S\ref{surjectivity}, (VII) can be rewritten as follows:
\begin{multline*}
0\to p^1 H_i(X)\to\lim\nolimits^1 Z_{i+1}^N\to pH_{i-1}(X)\to qH_{i-1}(X)\\
\to q^2H_i(X)\to\lim\nolimits^1 Z_i^N\to q^1H_{i-1}(X)\to 0.
\end{multline*}
Finally, the original exact sequence (V) implies that $\lim\nolimits^1 Z_{i+1}^N$ is isomorphic to
$K_{i-1}=\ker\big(H_{i-1}(X)\to qH_{i-1}(X)\big)$.
Thus we get 
\[\resizebox{\linewidth}{!}{
$0\to p^1 H_i(X)\to K_{i-1}\to pH_{i-1}(X)\to qH_{i-1}(X)
\to q^2H_i(X)\to K_{i-2}\to q^1H_{i-1}(X)\to 0.$}\]
\end{proof}

\begin{mainconjecture} \label{cjA}
If $X$ is a local compactum, and $N_\alpha$ are the nerves of its
open covers, then $\lim^2 H_n(N_\alpha)=0$ for each $n$.
\end{mainconjecture}

\begin{mainconjecture} \label{cjB}
If $X$ is a local compactum, then the homomorphism 
$\tau_n^X\:pH_n(X)\to qH_n(X)$ is surjective for each $n$.
\end{mainconjecture}

\begin{corollary} \label{lim2corollary}
Under the assumptions (1), (2) Conjecture \ref{cjA} implies Conjecture \ref{cjB}
and moreover that there is an exact sequence
\[0\to p^1 H_{i+1}(X)\to q^1H_{i+1}(X)\to pH_i(X)\xr{\tau_n^X} qH_i(X)\to 0.\]
\end{corollary}

\subsection*{Acknowledgements}

I would like to thank J. Bergfalk and F. Pakhomov for useful remarks.

\subsection*{Disclaimer}

I oppose all wars, including those wars that are initiated by governments at the time when 
they directly or indirectly support my research. The latter type of wars include all wars 
waged by the Russian state in the last 25 years (in Chechnya, Georgia, Syria and Ukraine) 
as well as the USA-led invasions of Afghanistan and Iraq.

\end{document}